\documentclass[a4paper,12pt]{article}
\usepackage{amsthm}
\usepackage{amsfonts}
\usepackage{amsmath}
\usepackage{amssymb}
\usepackage[pdftex]{graphicx}
\usepackage{hyperref}
\usepackage{xcolor}
\usepackage{breqn}

\numberwithin{equation}{section} 
\newtheorem{theorem}[equation]{Theorem}
\theoremstyle{definition}
\newtheorem{definition}[equation]{Definition}
\theoremstyle{remark}
\newtheorem{note}[equation]{Remark}
\theoremstyle{plain}
\newtheorem{lemma}[equation]{Lemma}
\newtheorem{corollary}[equation]{Corollary}
\newtheorem{proposition}[equation]{Proposition}

\newcommand{\diag}{\operatorname{diag}}
\newcommand{\bq}{/\!\!/}

\newcommand{\re}{\operatorname{Re}}
\newcommand{\im}{\operatorname{Im}}
\newcommand{\Tr}{\operatorname{Tr}}

\title{Almost positively curved generalized Eschenburg spaces}
\author{Jason DeVito and Joan West}
\date{}

\begin{document}

\maketitle

\begin{abstract} In each dimension of the form $4n-1$ with $n\geq 3$, we construct infinitely many new examples of manifolds admitting metrics with positive sectional curvature almost everywhere.  In addition, we show that if $n\geq 6$, infinitely many of our examples are not homotopy equivalent to any homogeneous space, providing the first infinite family of such examples.
\end{abstract}

%%%%%%%%%%%%%%%%%%%%%%%%%%%%
\section{Introduction}
%%%%%%%%%%%%%%%%%%%%%%%%%%%%

 The study of non-negative and positive sectional curvature is essentially as old as Riemannian geometry itself.  Compact rank one symmetric spaces (CROSSes), i.e., spheres and the projective spaces over $\mathbb{C},\mathbb{H}$, and $\mathbb{O}$ are well-known examples of simply connected closed manifolds admitting metrics of positive sectional curvature.  But apart from these, examples are quite scarce.  In fact, in dimension above 24, all known examples are diffeomorphic to a CROSS.

By comparison, non-negatively curved examples are much more common, including all Riemannian homogeneous spaces and all cohomogeneity one manifolds with codimension two singular orbits \cite{GZ}, and their isometric quotients.

In this article, we are concerned with Riemannian manifolds interpolating between these two classes.  We recall that a Riemannian manifold is said to be \textit{quasi-positively curved} if it has non-negative sectional curvature and a point at which all two-planes are positively curved.  A Riemannian manifold is called \textit{almost positively curved} if the set of points at which all two-planes are positively curved is open and dense.  While examples of Riemannian manifolds of almost positive curvature are more abundant than examples with positive sectional curvature, most of them are obtained by deforming a homogeneous metric.  In fact, apart from finitely many low-dimensional examples \cite{longkerinpaper,EK,DeV2}, the only known inhomogeneous examples are due to Wilking \cite{Wi} and are diffeomorphic to biquotients of the form $\Delta SO(2) \backslash SO(2n+1)/SO(2n-1), \Delta SU(2)\backslash SO(4n+1)/SO(4n-1)$, and to $\Delta Sp(1)\backslash Sp(n+1)/(Sp(1)\cdot Sp(n-1)$.  In particular, in each fixed dimension, there are currently only finitely many known inhomogeneous examples.

Our first main result is the existence of infinitely many new strongly inhomogeneous almost positively curved examples in each dimension of the form $4n-1$ with $n\geq 6$.  

\begin{theorem}\label{thm:main} In each dimension of the form $4n-1$ with $n\geq 3$, there are infinitely many simply connected closed manifolds which admit a metric of almost positive sectional curvature.  In addition, if $n\geq 6$, infinitely many of them are strongly inhomogeneous.
\end{theorem}

We recall that a manifold is said to be \textit{strongly inhomogeneous} if it is not homotopy equivalent to any homogeneous manifold.  The manifolds of Theorem \ref{thm:main} are all obtained as free isometric circle quotients of the homogeneous space $U(n+1)/U(n-1)$, where $U(n+1)/U(n-1)$ is equipped with a particular homogeneous metric first constructed by Wilking \cite{Wi}.  Wilking used this metric to endow infinitely many of the homogeneous circle quotients with an almost positively curvature.  Later, Tapp \cite{Ta} showed Wilking's metric endowed all of the homogeneous circle quotients with quasi-positive curvature.  Kerin \cite{longkerinpaper} was the first to study the inhomogeneous circle quotients, where he used a different metric on $U(n+1)/U(n-1)$ to endow all of these quotients with a metric of quasi-positive curvature.  Each of these quotients is termed a generalized Eschenburg space, since one obtains the classical Eschenburg spaces in the special case that $n=2$.  Eschenburg spaces were first constructed by Eschenburg \cite{Eschenburg}, where he showed an infinite subfamily of them admit metrics of positive sectional curvature.  When $n=2$, Shankar \cite{Sh} has shown that infinitely many of them are strongly inhomogeneous as well.

In fact, Theorem \ref{thm:main} is a special case of our second main result which covers all generalized Eschenburg spaces whose natural isometry group acts via cohomogeneity two.  It turns out that such spaces are described by three integers $p,q_1,q_2$ satisfying $\gcd(q_1,q_2) = \gcd(p-q_1,q_2)=\gcd(q_1,p-q_2) = 1$, and that one may assume without loss of generality that either $q_2 > 0$ or that $q_2 = 0$ and $q_1 > 0$.

\begin{theorem}\label{thm:main2} For any allowable integers $(p,q_1,q_2)$ other than $(0,0,1)$, $(0, 1, 0)$, and $(1,1,0)$  the generalized Eschenburg space $\mathcal{E}^{4n-1}_{p,q_1,q_2}$ with $n\geq 2$ has almost positive curvature with respect to Wilking's metric if and only if one of the following occurs:

\begin{itemize} \item $0, p \in [\min\{q_1+q_2,q_2\}, \max\{q_1+q_2,q_2\}]$

\item  $p\geq q_1 + q_2 > 0$ and $q_1\geq 0$

\end{itemize}

In addition, for the exceptional triples, the Wilking metric is not almost positively curved.
\end{theorem}

% joan question: why brackets, in case it's a multiset?  No, I meant interval notation: $[a,b] = \{x: a\leq x \leq b\}$.

The idea of the proof is as follows.  Apart from a few special cases, to each admissible triple $(p,q_1,q_2)$, we associate a polynomial $f_{p,q_1,q_2}:\mathbb{R}^2\rightarrow \mathbb{R}$, of degree $4$ in each variable, with the property that $\mathcal{E}_{p,q_1,q_2}$ is almost positively curved if and only if $f_{p,q_1,q_2} \leq 0$ on $[0,1]\times [0,1]$.  We then analyze these polynomials via a variety of tools.

In the special case where $q_1 = q_2$ (where admissibility implies $q_1 = q_2 =1)$, the corresponding generalized Eschenburg space actually admits a (non-isometric) cohomogeneity one action.  Wulle \cite{Wu} has shown that when $4n-1 > 7$, no generalized Eschenburg space admits a cohomogeneity one -invariant metric of quasi-positive curvature.

\begin{corollary}  In each dimension of the form $4n-1$ with $n\geq 3$, there are infinitely cohomogeneity one manifolds which do not admit a cohomogeneity one metric of quasi-positive curvature but do admit a cohomogeneity two metric of almost positive curvature.
\end{corollary}

We also study free isometric $T^2$ actions on $U(n+1)/U(n-1)$, equipped with Wilking's metric.  Our main result here is the following.

\begin{theorem}\label{thm:even}  For each $n\geq 2$, there are at least three free isometric $T^2$-actions on $U(n+1)/U(n-1)$.  All three of the resulting quotient spaces inherit almost positively curved metrics.

\end{theorem}

Wilking has already proved this for one of the three quotients: the projectivized tangent bundle to $\mathbb{C}P^n$.

We now outline the structure of this article.  In Section \ref{sec:back}, we cover the relevant background on generalized Eschenburg spaces.  In Section \ref{sec:cohom}, we show that the generalized Eschenburg spaces we consider admit cohomogeneity two actions, and we determine a suitable two-dimensional subset intersecting each orbit.  In Section \ref{sec:easycurv}, we begin the analysis of zero-curvature planes, showing that for most ``types", they can only occur on nowhere dense subsets.  Section \ref{sec:poly} is devoted to constructing the two-variable, degree four polynomial mentioned above.  Then, in Section \ref{sec:polyanalysis}, we analyze this polynomial and complete the proof of Theorem \ref{thm:main2}.  Section \ref{sec:t2} contains the proof of Theorem \ref{thm:even}.  In Section \ref{sec:quasi}, we use a proof provided by Kerin to fill a small gap in Kerin's result \cite[Theorem B]{longkerinpaper} that every generalized Eschenburg space admits a quasi-positively curved metric.  Finally, in Section \ref{sec:top}, we show that infinitely many of the almost positively curved generalized Eschenburg spaces are strongly inhomogeneous.

Both authors gratefully acknowledge support from NSF DMS 2405266.  They also want to thank Martin Kerin for helpful conversations, as well as suggesting the proof of Proposition \ref{prop:fix}.

\section{Background}\label{sec:back}

\subsection{The construction of generalized Eschenburg spaces}\label{sec:back1}

In this section, we collect background regarding generalized Eschenburg spaces and the metrics we consider on them.

Let $G = U(n+1)$ and suppose $\vec{p} = (p_1,...,p_{n+1})\in \mathbb{Z}^{n+1}$ and $\vec{q} = (q_1,q_2)\in \mathbb{Z}^2$ with $\gcd(p_1,p_2,...,p_{n+1},q_1,q_2) = 1$.  The group $L = U(n-1)\times S^1$ then acts on $G$ as $(C,z)\bullet B= \diag(z^{p_1},...,z^{p_{n+1}}) B \diag(z^{q_1},z^{q_2}, C)^{-1})$.  As is well known, this action is free if and only if  \begin{equation}\label{eqn:gcd} \gcd(p_i - q_1, p_j - q_2) = 1 \text{ for all } i\neq j \in \{1,...,n+1\}.\end{equation}

By precomposing with the map $z\mapsto \overline{z}$, we obtain an equivalent action with $q_2$ replaced by $-q_2$.  As such, we can and will restrict our attention to actions for which $q_2\geq 0$.  In addition, if $q_2 = 0$, we can similarly assume that $q_1 > 0$.

\begin{definition}\label{def:admis} A collection of integers $p_1,...,p_{n+1}, q_1,q_2$ is called \textit{admissible} if \eqref{eqn:gcd} holds and either $q_2 > 0$ or both $q_2 = 0$ and $q_1 > 0$.
\end{definition}

When the action is free, the quotient space $G\bq L$ is denoted $\mathcal{E}_{\vec{p},\vec{q}}$ and is a manifold, called a generalized Eschenburg space.  In the special case where $n=2$, the resulting biquotients are the classical Eschenburg spaces.  Eschenburg spaces were introduced by Eschenburg in \cite{Eschenburg}, where it was shown that an infinite subfamily of them admits a metric of positive sectional curvature. 
Later, Kerin \cite[Theorem A(i)]{shortkerinpaper} showed that Eschenburg's metric can be used to endow all $7$-dimensional Eschenburg spaces (except one) with a metric of quasi-positive curvature---that is, a non-negatively curved metric for which at least one point has positive curvature.  The exceptional example was shown by Wilking \cite[Theorem 1c)]{Wi}  to admit a different metric of almost positive curvature---that is, it is positively curved on an open and dense set.  Later, the first author and Johnson \cite[Theorem 1.2]{DJ} characterized the curvature (strictly positive, almost positive, quasi-positive) of each Eschenburg space with respect to Eschenburg's metric.

Generalized Eschenburg spaces have also received much attention.  Wilking \cite[Theorem 1c)]{Wi} showed that in the special case that $\vec{p} = \vec{0}$ and $q_1 q_2 < 0$, $\mathcal{E}_{\vec{p},\vec{q}}$ admits a metric of almost positive sectional curvature.  More precisely, he showed that after possibly switching $q_1$ and $q_2$ (which yield diffeomorphic but non-isometric spaces), that his metric on $U(n+1)/U(n-1)$ induces almost positive curvature on $\mathcal{E}_{\vec{0},\vec{q}}$.  Tapp \cite[Theorem 1.1 (2)]{Ta} showed that under the weaker assumption that $\vec{p} = \vec{0}$ and $(q_1,q_2)\neq (0,0)$ one can similarly obtain the weaker conclusion that $\mathcal{E}_{\vec{0},\vec{q}}$ is quasi-positively curved.  Kerin \cite[Theorem B]{longkerinpaper} then generalized Tapp's result: by modifying the metric Tapp and Wilking used, Kerin endowed all $\mathcal{E}_{\vec{p},\vec{q}}$ with a quasi-positively curved metric.  In the following section, we describe properties of these metrics on much more detail.

\subsection{Metrics on generalized Eschenburg spaces}\label{sec:eschmetric}  In this section, we will give the details of both Wilking and Kerin's metric constructions, beginning with Wilking's. His metric is obtained as follows.  First, one equips $G\times G$ with a product metric, where the metric on each factor is a Cheeger deformation of a bi-invariant metric in the direction of $K = U(1)\cdot U(n)\subseteq G$.  Viewing $L\subseteq G\times G$ as $$L = \{(\diag(z^{p_1},..., z^{p_2}), \diag(z^{q_1}, z^{q_2}, C)): z\in S^1, C\in U(n-1)\},$$ this metric is invariant under the free action of $G\times L$ defined by $$(g,(\ell_1,\ell_2))\cdot (g_1,g_2) = (gg_1 \ell_1^{-1}, gg_2 \ell_2^{-1}).$$

In particular, it induces a metric on the quotient $\Delta G\backslash (G\times G)/L$.  This quotient is diffeomorphic to $G\bq L$ via  diffeomorphism induced via the map $\mu:G\times G\rightarrow G$ with $\mu(g_1,g_2) = g_1^{-1} g_2$, so one obtains a Riemannian metric on each generalized Eschenburg space.  We refer to this metric as Wilking's metric.

Metrics obtained via this kind of construction have a relatively simple characterization of points with zero-curvature planes.  To state it in our context, we first set up notation.  Let $\mathfrak{g}$ and $\mathfrak{l}$ denote the Lie algebras of $G$ and $L$ respectively.  Write $\mathfrak{g} = \mathfrak{k}\oplus \mathfrak{p}$, where $\mathfrak{k}$ is the Lie algebra of $K$ and $\mathfrak{p}$ is an orthogonal complement to $\mathfrak{k}$ with respect to the $Ad_G$-invariant inner product $\langle X,Y\rangle_0 = -\Tr(XY)$ on $\mathfrak{g}$.  Finally, let $\phi_1:\mathfrak{g}\rightarrow\mathfrak{g}$ be the linear map defined by $\phi_1(X) = \lambda_1X_{\mathfrak{k}}+X_{\mathfrak{p}}$ where $\lambda_1 = \frac{t_1}{t_1+1}$, $t_1\in (0,\infty)$ is a fixed parameter, and $X = X_{\mathfrak{k}}+X_{\mathfrak{p}}$ with $X_{\mathfrak{k}}\in \mathfrak{k}$ and $X_{\mathfrak{p}}\in \mathfrak{p}$.

\begin{proposition}\label{prop:Wcurvature}  With respect to Wilking's metric, there is a zero-curvature plane at $[(B,I)]\in \Delta G \backslash(G\times G)/L\cong \mathcal{E}_{\vec{p},\vec{q}}$ if and only if there are linearly independent vectors $X,Y\in \mathfrak{g}$ satisfying all of the following properties:

\begin{enumerate}\item  $\langle X, Ad_B L_1 - L_2\rangle_0 = \langle Y, Ad_B L_1-L_2\rangle_0$ for all $(L_1,L_2)\in \mathfrak{l}\subseteq \mathfrak{g}\oplus \mathfrak{g}$.

\item $[X,Y]= 0$

\item  $\{X_{\mathfrak{p}}, Y_{\mathfrak{p}}\}$ is linearly dependent.

\item  $\{ (Ad_{B^{-1}}(X))_{\mathfrak{p}}, (Ad_{B^{-1}}(Y))_{\mathfrak{p}} \}$ is linearly dependent.

\end{enumerate}

\end{proposition}

\begin{proof}Suppose initially that each of these conditions is satisfied.  Because $X$ and $Y$ satisfy the first equation, via \cite[Equation (9), pg. 1006]{longkerinpaper}, the span of the vectors $\widehat{X} = (-\phi_1^{-1}(Ad_{B}^{-1} X), \phi_1^{-1}(X))$ and $\widehat{Y} = (-\phi_1^{-1}(Ad_{B}^{-1} Y), \phi_1^{-1}(Y)) \in \mathfrak{g}\oplus \mathfrak{g}$ lefts translates to a horizontal plane $\widehat{\sigma}$ at $(B,I)\in G\times G$.  We claim that this plane has zero-curvature.  Believing this, it follows from Tapp's theorem \cite{Ta2} that the projection of this plane to $[(B,I)]\in \Delta G\backslash(G\times G)/L$ has zero-curvature.

Because the metric on $G\times G$ is a product of non-negatively curved metrics, to verify $\widehat{\sigma}$ has zero-curvature, it is sufficient to verify that the projection of $\widehat{\sigma}$ to each factor of $G\times G$ has zero-curvature.  The projection to the first factor is $\widehat{\sigma}_1 = \operatorname{span} \{ \phi_1^{-1} Ad_{B^{-1}}X, \phi_1^{-1} Ad_{B^{-1}} Y\}$ and the projection to the second factor is $\widehat{\sigma}_2 =\operatorname{span}\{ \phi_1^{-1} X, \phi_1^{-1}Y\}$.

We focus on $\widehat{\sigma}_2$ first.  From, e.g., \cite[Example (a), pg. 1005]{longkerinpaper}, we see that since $(G,K)$ is a symmetric pair, that $\widehat{\sigma}_2$ has zero-curvature if and only if $[X,Y] = [X_{\mathfrak{k}}, Y_{\mathfrak{k}}] = [X_{\mathfrak{p}},Y_{\mathfrak{p}}] = 0$.  By assumption, $[X,Y] = [X_{\mathfrak{p}}, Y_{\mathfrak{p}}] = 0$, so we only need to verify that $[X_{\mathfrak{k}},Y_{\mathfrak{k}}] = 0$.  However, since $[\mathfrak{k},\mathfrak{p}]\subseteq \mathfrak{p}$ and $[X_{\mathfrak{p}},Y_{\mathfrak{p}}] =0$, it follows that the $\mathfrak{k}$-component of the equation $0 = [X,Y]$ is $[X_{\mathfrak{k}},Y_{\mathfrak{k}}]$, so this must be zero as well.  Thus, $\widehat{\sigma}_2$ has zero-curvature.

Turning to $\widehat{\sigma}_1$, \cite[Example (a), pg. 1005]{longkerinpaper} implies that $\widehat{\sigma}_1$ has zero-curvature if and only if $[Ad_{B^{-1}}X, Ad_{B^{-1}}Y] = [(Ad_{B^{-1}}X)_{\mathfrak{k}}, (Ad_{B^{-1}}Y)_{\mathfrak{k}}] = [(Ad_{B^{-1}}X)_{\mathfrak{p}}, (Ad_{B^{-1}}Y)_{\mathfrak{p}}] = 0$.  The fact that $[X,Y]=0$ and that $Ad_{B^{-1}}$ is a Lie algebra isomorphism implies that $[Ad_{B^{-1}}X, Ad_{B^{-1}}Y] = 0$.  Further, $[(Ad_{B^{-1}}X)_{\mathfrak{p}}, (Ad_{B^{-1}}Y)_{\mathfrak{p}}]=0$ by assumption, and then the fact that $[(Ad_{B^{-1}}X)_{\mathfrak{k}}, (Ad_{B^{-1}}Y)_{\mathfrak{k}}] = 0$ now follows as it did in the case of $\widehat{\sigma}_2$.

Thus, $\widehat{\sigma}$ has zero-sectional curvature, so there is a zero-curvature plane at $[(B,I)]$.

\bigskip

Conversely, assume there is a zero-curvature plane at $[(B,I)]$.  Then from the Gray-O'Neill formulas for a Riemannian submersion \cite{Gr,On1} there must be a horizontal zero-curvature plane at $(B,I)$.   From \cite[Equation (9), pg. 1006]{longkerinpaper}, left translating this plane the identity yields a horizontal zero-curvature plane of the form $\widehat{\sigma}=\operatorname{span} \{\widehat{X},\widehat{Y}\}$ where $X$ and $Y$ satisfy the first condition of this proposition.

Since $\widehat{\sigma}$ has zero-curvature, its two projections $\widehat{\sigma}_i$ must have zero-curvature.  From \cite[Example (a) pg. 1005]{longkerinpaper}, we find that $[X,Y] = 0$ and that both $[X_{\mathfrak{p}}, Y_{\mathfrak{p}}] = 0$ and $[(Ad_{B^{-1}} X)_{\mathfrak{p}}, (Ad_{B^{-1}} Y)_{\mathfrak{p}}] = 0$.  However, since the $\mathfrak{p}$-components can be interpreted as tangent vectors in the positively-curved $G/K = \mathbb{C}P^{n}$, their bracket vanishes if and only if they are linearly dependent.

\end{proof}

In \cite{longkerinpaper}, Kerin uses a slightly different metric.  Specifically, he endows $G\times G$ with a product metric where the metric on the first factor is a Cheeger deformation in the direction of $K$, but the metric on the second factor is a double Cheeger deformation---first in the direction of $K$ and then in the direction of $H:= U(1)\times U(1)\times U(n-1)$.  We let $\Psi:\mathfrak{g}\rightarrow \mathfrak{g}$ be defined by $\Psi(X) = \lambda_2 X_{\mathfrak{h}} + X_{\mathfrak{m}}$  where $\lambda_2 = \frac{t_2}{t_2+1}$, $t_2\in (0,\infty)$ is a fixed parameter, and $\mathfrak{g} = \mathfrak{h}\oplus \mathfrak{m}$ is orthogonal decomposition with respect to $\langle \cdot, \cdot \rangle_0$.  We similarly let $\phi_2(X) = t_1t_2 X_{\mathfrak{h}} + t_1 X_{\mathfrak{m}} + X_{\mathfrak{p}}$, so that $\Psi = \phi_1^{-1} \phi_2$.

We again have a good characterization of points with zero-curvature planes with respect to Kerin's metric.  The proof is analogous to the proof of Proposition \ref{prop:Wcurvature}, except that one uses $\widehat{X} = (-\phi_1^{-1} (Ad_{B^{-1}}(\phi_1(X))), \Psi^{-1} X)$ and  \cite[Example (b), pg. 1005]{longkerinpaper}.

\begin{proposition}\label{prop:Kcurvature} With respect to Kerin's metric, there is a zero-curvature plane at $[(B,I)]\in \Delta G \backslash(G\times G)/L\cong \mathcal{E}_{\vec{p},\vec{q}}$ if and only if there are linearly independent vectors $X,Y\in \mathfrak{g}$ satisfying all of the following properties:

\begin{enumerate}\item  $\langle \phi_1(X), Ad_B L_1 - L_2\rangle_0 = \langle \phi_1(Y), Ad_B L_1-L_2\rangle_0$ for all $(L_1,L_2)\in \mathfrak{l}\subseteq \mathfrak{g}\oplus \mathfrak{g}$.

\item $[X,Y]= 0$

\item  $\{X_{\mathfrak{p}}, Y_{\mathfrak{p}} \}$ is linearly dependent.

\item  $\{ (Ad_{B^{-1}}\phi_1(X))_{\mathfrak{p}}, (Ad_{B^{-1}}\phi_1(Y))_{\mathfrak{p}}\}$ is linearly dependent.

\item  $[X_{\mathfrak{m}},Y_{\mathfrak{m}}] = [X_{\mathfrak{h}}, Y_{\mathfrak{h}}] = 0$

\end{enumerate}

\end{proposition}

We will use Kerin's classification of horizontal zero-curvature planes given in  \cite[Proposition 5.1]{longkerinpaper}. Note, however, that the type (iv) plane in our statement differs slightly from that in [10, Proposition 5.1].  This is due to a small error in Kerin's proof where he asserts that a type (iv) plane as stated in Proposition \ref{prop:martin5.1} can be assumed to have $\alpha = 0$ because it may be assumed, without loss of generality, that the vectors $\Psi^{-1}(X)$ and $\Psi^{-1}(Y)$ below are orthogonal.  However, if $\alpha \neq 0$, then an orthogonal basis for a plane of type (iv) is obtained by modifying $\Psi^{-1}(X)$ by a multiple of $\Psi^{-1}(Y)$ and resulting vector $X$ is non-zero in the bottom $n\times n$ block, so no longer belongs to type (iv).  In Theorem \ref{thm:quasi}, whose proof was provided by Kerin, we will show that the main result of \cite{longkerinpaper} is correct in spite of this error.

\begin{proposition}[Kerin]\label{prop:martin5.1}  With respect to Kerin's metric, assume there is a horizontal zero-curvature plane at $(B,I)\in G\times G$.  Then the left translation of this plane to the identity is spanned by vectors   $$\widehat{X} = (-\phi_1^{-1} (Ad_{B^{-1}}(\phi_1(X))), \Psi^{-1} X) \text{ and } \widehat{Y}  = (-\phi_1^{-1} (Ad_{B^{-1}}(\phi_1(Y))), \Psi^{-1} Y)$$ where $X$ and $Y$ can be expressed as one of the following types: 

    \begin{itemize}
        \item[(i)] $X\in \mathfrak{g}$ and $Y \in \diag(i, 0, ..., 0)$
        
        \item[(ii)] $X\in \mathfrak{p} \oplus \mathfrak{h}$ and $Y \in \diag(0, i, ..., 0)$
        
        \item[(iii)] $X\in \mathfrak{p} \oplus \mathfrak{h}$ and $Y \in \diag(i, i, ..., 0)$
        
        \item[(iv)] 
        $X = \left(\begin{tabular}{c|ccc}
            $i\alpha$ & & $-\overline{x^T}$ & \\
            \hline
             & & & \\
            $x$ & & 0 & \\
             & & &
        \end{tabular}\right)$ and $Y = \left(\begin{tabular}{c|c|ccc}
             $i$ & & & & \\
             \hline
             & $i\beta$ & & $-\overline{y^T}$ & \\
             \hline
             & & & & \\
             & $y$ & & 0 & \\
             & & & &
        \end{tabular}\right)$

        where $x_2 \not= 0$, $\beta = 1 - \sum_{j=3}^{n+1}|y_j|^2 $, and $x_j=-ix_2y_j$ for $j\in \{3,...,n+1\}$
        
        \item[(v)]
        $X = \left(\begin{tabular}{c|ccc}
            $i\alpha$ & & $-\overline{x^T}$ & \\
            \hline
             & & & \\
            $x$ & & 0 & \\
             & & &
        \end{tabular}\right)$ and $Y = \left(\begin{tabular}{c|c|ccc}
             $0$ & & & & \\
             \hline
             & $i\beta$ & & $-\overline{y^T}$ & \\
             \hline
             & & & & \\
             & $y$ & & 0 & \\
             & & & &
        \end{tabular}\right)$

        where $x =  (0, x_3, ... , x_{n+1})^T \not= 0$ and $\sum_{j=3}^{n+1}x_j\overline{y_j} = 0 $.
    \end{itemize}

\end{proposition}

The next proposition indicates that for the task of determining whether a point has a zero-curvature plane, we can freely switch between Wilking's metric and Kerin's metric.

\begin{proposition}\label{prop:switchmetric} For any $[(B,I)]\in \mathcal{E}_{\vec{p},\vec{q}}$, there is a zero-curvature plane at $[(B,I)]$ with respect to Kerin's metric if and only if there is a zero-curvature plane there with respect to Wilking's metric.

\end{proposition}

\begin{proof}  Assume initially that there is a zero-curvature plane with respect to Wilking's metric.  From Proposition \ref{prop:Wcurvature}, there are linearly independent vectors $X,Y\in \mathfrak{g}$ satisfying all four conditions of Proposition \ref{prop:Wcurvature}.

We claim that $X':=\phi_1^{-1}(X) = \frac{1}{\lambda_1} X_{\mathfrak{k}} + X_{\mathfrak{p}}$ and $Y' = Y$ satisfy all 5 conditions of Proposition \ref{prop:Kcurvature}.  Indeed, since in all 5 types of Proposition \ref{prop:martin5.1} have $Y\in \mathfrak{k}$, it follows that $\phi_1(Y) = \lambda_1 Y$ is a multiple of $Y$.  In particular, conditions 1 and 4 of Proposition \ref{prop:Kcurvature} follow immediately from conditions 1 and 4 of Proposition \ref{prop:Wcurvature}.  Moreover, since $ X'$ and $X$ have the same $\mathfrak{p}$ component, the same holds for condition 3.

For condition 2, we argue as follows.  We already know that $0 = [X,Y]$ and since $Y\in \mathfrak{p}$ and $(G,K)$ is a symmetric pair, we see that $0 = \underbrace{[X_{\mathfrak{k}}, Y]}_{\mathfrak{p}} + \underbrace{[X_{\mathfrak{p}},Y]}_{\mathfrak{k}}$ so that $[X_{\mathfrak{k}},Y] =[X_{\mathfrak{p}},Y]= 0$.  Then $$[X',Y'] = \frac{1}{\lambda_1} [X_{\mathfrak{k}}, Y] + [ X_{\mathfrak{p}}, Y] = 0,$$ confirming condition 2 of Proposition \ref{prop:Kcurvature}.

We finally verify condition $5$.  First, the condition 1 implies that $Y$ is orthogonal to $\mathfrak{h}$, so that $Y_{\mathfrak{h}} = 0$.  Thus, $[X'_{\mathfrak{h}}, Y'_{\mathfrak{h}}] = [X'_{\mathfrak{h}},0] = 0$.  On the other hand, in types (i), (ii), and (iii), $Y$ has vanishing $\mathfrak{m}$ component, while in types (iv) and (v), $X$ has vanishing $\mathfrak{m}$ component.  Since $X'_{\mathfrak{m}} = \lambda_1 X_{\mathfrak{m}}$, we see that in all five types, either $X'_{\mathfrak{m}} = 0$ or $Y'_{\mathfrak{m}} = 0$, so that $[X'_{\mathfrak{m}}, Y'_{\mathfrak{m}}] = 0$ in all cases.

Having verified all $5$ conditions of Proposition \ref{prop:Kcurvature}, it follows that there is a zero-curvature plane with respect to Kerin's metric.

Conversely, if $X$ and $Y$ satisfy all 5 conditions of Proposition \ref{prop:Kcurvature}, then arguing similarly, it is easy to see that $X' = \phi_1(X)$ and $Y' = Y$ satisfy all 4 conditions of Proposition \ref{prop:Wcurvature}.

\end{proof}

We have the following characterization of zero-curvature planes of type (iv) and (v).

\begin{proposition}\label{prop:pregiveseqns}  Suppose $X$ and $Y$ fall into type (iv) or (v) of Proposition \ref{prop:martin5.1} and let $\sigma = \operatorname{span}\{\widehat{X}, \widehat{Y}\}$.  Then the left translation of $\sigma$ to $(B,I)$ is a horizontal zero-curvature plane with respect to Kerin's metric if and only if all of the following conditions are satisfied.

\begin{enumerate}\item
$\Tr(i\phi_1(X)  (Ad_B (\diag(p_1,p_2,..., p_{n+1})) - \diag(q_1,q_2,0,...,0)) = 0$

\item $\Tr(i Y  (Ad_B (\diag(p_1,p_2,..., p_{n+1})) - \diag(q_1,q_2,0,...,0)) = 0$

\item $\{(Ad_{B^{-1}}(\phi_1(X)))_{\mathfrak{p}}, (Ad_{B^{-1}} Y)_{\mathfrak{p}}\}$ is linearly dependent.

\end{enumerate}

\end{proposition}

\begin{proof}From Proposition \ref{prop:Kcurvature}, $\sigma$ has zero-curvature if and only if all five conditions of Proposition \ref{prop:Kcurvature} hold.

Noting that all five conditions are scale invariant, and that $\phi_1(Y)$ is a multiple of $Y$ since $Y\in \mathfrak{k}$, we may replace all occurrences of $\phi_1(Y)$ with $Y$.  Thus, condition 4 of Proposition \ref{prop:Kcurvature} yields the third condition of this proposition.

We also note that by inspection, conditions $2,3,$ and $5$ hold automatically and that condition $1$ holds for all $(L_1,L_2)\in \mathfrak{u}(n-1)\subseteq \mathfrak{l} = \mathfrak{u}(1)\oplus \mathfrak{u}(n-1)$.  Thus condition one reduces to checking orthogonality to the $\mathfrak{u}(1)$ factor.  But clearly $\mathfrak{u}(1) = \operatorname{span} \{ i (\diag(p_1,...., p_{n+1}), \diag(q_1,q_2,0,...,0))\}$.  Thus, condition 1 of Proposition \ref{prop:Kcurvature} gives conditions one and two of this proposition.
\end{proof}

%%%%%%%%%%%%%%%%%%%%%%%%%%%%
\section{Cohomogeneity two generalized Eschenburg spaces}\label{sec:cohom}
%%%%%%%%%%%%%%%%%%%%%%%%%%%%

We now turn our attention to a family of generalized Eschenburg spaces admitting a natural action of cohomogeneity two.  Specifically, we will assume that $\vec{p} = (p_1,p_2 , p_2,.... , p_2)$.  If we subtract the same number simultaneously from all entries of both $\vec{p}$ and $\vec{q}$, the resulting $S^1$-action on $U(n+1)/U(n-1)$ is the same.  Hence, we may assume without loss of generality that $p_2 = p_3 = ... = p_{n+1} = 0$, and we write $p$ in place of $p_1$.  We will sometimes denote the resulting space $\mathcal{E}_{\vec{p},\vec{q}}$ as $\mathcal{E}_{p, q_1,q_2}$.

From equation \eqref{eqn:gcd}, we have the following characterization of when the action is free.

\begin{proposition}\label{prop:simpad}  A collection $p, 0, ..., 0, q_1,q_2$ of integers is admissible if and only if \begin{equation}\label{eqn:gcd2}\gcd(q_1,q_2) = \gcd(p-q_1,q_2) = \gcd(q_1,p-q_2) = 1.\end{equation}
\end{proposition}

If we specialize Definition \ref{def:admis} to the case of three integers, we obtain the following:

\begin{definition} A triple of integers $(p,q_1,q_2)$ is called admissible if equation \eqref{eqn:gcd2} holds and if $q_2 > 0$ or $(q_1,q_2) = (1,0)$.
\end{definition}

We now show each $\mathcal{E}_{p,q_1,q_2}$ admits a natural cohomogeneity two action.  To that end, we let $A(t,r)$ be defined as the matrix $$A(t,r)= \begin{bmatrix} \cos t \cos r & -\sin r & -\sin t\cos r\\ \cos t\sin r & \cos r & -\sin t \sin r\\ \sin t & 0 & \cos t\end{bmatrix}.$$

Set $$\mathcal{F} = \{(\diag(A(t,r),I_{n-2}), I_{n+1})\in G\times G: t,r\in [0,\pi/2]\}.$$   We observe that $\mathcal{F}$ is homeomorphic to a two-dimensional ball, so $\dim \mathcal{F} = 2$.  By abuse of notation, we will sometimes identify the element $\diag(A(t,r),I_{n-2}, I_{n+1})\in \mathcal{F}$ with $A(t,r)\in SO(3)$.

\begin{proposition}\label{prop:cohom2} For any admissible triple $(p,q_1,q_2)$, the action by $U(n)\times T^3$ on $G\times G$ given by $$(D,w_1,w_2,w_3)\ast (g_1,g_2) = (g_1 \diag(w_1,D)^{-1}, g_2 \diag(w_2,w_3,I_{n-1})^{-1})$$ descends to an action on $\mathcal{E}_{p,q_1,q_2}$.  Every $(U(n)\times T^3)$-orbit intersects the image of $\mathcal{F}$ in exactly one point, so that, in particular, the action has cohomogeneity two.

\end{proposition}

\begin{proof}
We first observe that for $z\in S^1$, $\diag(z^{p_1},...,z^{p_n}) = \diag(z^{p}, 1,...,1)$ and thus $\diag(z^p,1,...,1)$ commutes with $\diag(w_1,D)$.  In a similar fashion, $\diag(z^{q_1}, z^{q_2}, C)$ with $C\in U(n-1)$, commutes with $\diag(w_2, w_3, I_{n-1})$.  It follows that that the $(U(n)\times T^3)$-action given by $\ast$ commutes with the action $\cdot$ of $G\times L$, described in Section \ref{sec:eschmetric}. In particular, the $(U(n)\times T^3)$-action descends to an action on $\mathcal{E}_{p,q_1,q_2}$.  In fact, viewing $\mathcal{E}_{p,q_1,q_2}$ as the quotient $G\bq L$, the induced action by $U(n)\times T^3$ takes the form  $$(D,w_1,w_2,w_3) \ast [B] = [ \diag(w_1,D) B \diag(w_2,w_3,I_{n-1})].$$

We next show that every $(U(n)\times T^3)$-orbit intersects in image of $\mathcal{F}$ in at least one point.  It is clearly sufficient to show that every orbit of the $((U(n)\times T^3)\times (G\times L))$-action on $G\times G$ intersects $\mathcal{F}$.  We first observe that via the $G\times L$ action by $\cdot $ with $(g, (C,z)) = (g_2^{-1}, (I_{n-1}, 1))$, every orbit contains a point of the form $(g_1,I_{n+1})$.  Moreover, the stabilizer of the set $(G\times \{ I_{n+1}\})\subseteq G\times G$ under the $((U(n)\times T^3)\times (G\times L))$-action is given by $$J:= \{((D,w_1,w_2,w_3), (\diag(w_2 z^{q_1}, w_3z^{q_2}, C), (C,z))\}.$$  The induced $J$-action on $G\cong G\times \{I_{n+1}\}$ is given by $$g_1\mapsto \diag(w_2z^{q_1}, w_3 z^{q_2}, C) g_1 \diag( z^{p} w_1, D)^{-1}.$$  We will show that this $J$-action can move any point in $G$ to one in $\mathcal{F}$.

Write the entries of the matrix $g_1$ as $(g_1)_{ij}$.  Since the standard $U(n-1)$-action on $\mathbb{C}^{n-1}$ is transitive on the unit sphere, by choosing $C$ appropriately, and selecting $z=1, D=I_n$, $w_2 = \overline{(g_1)_{11}}$, and
$w_3 = \overline{(g_1)_{21}}$, we find that the orbit through $g_1$ contains a point whose first column is $$ v:=\begin{bmatrix} |(g_1)_{11}|\\ |(g_1)_{21}|\\ | ((g_1)_{31},...,(g_1)_{n+1,1})|\\ 0 \\ \vdots\\ 0\end{bmatrix}\in \mathbb{C}^{n+1}. $$  This column has unit length which implies that there are unique $t,r\in [0,\pi/2]$ for which the first three entries are $(\cos t \cos r, \cos t\sin r, \sin t)^T$. Thus, we assume without loss of generality that $g_1$ has a first column of the form of $\mathcal{F}$.

If we now further restrict the $J$-action to the action by $U(n)$ by only allowing $D$ to be a non-identity element, this restricted action will preserve the form of $g_1$.  If $g_1'$ is another element of $U(n+1)$ with the same first column as $g_1$, then the element $g_1^{-1} g_1'$ maps the standard basis vector $(1,0,...,0)^T$ to itself, so that $g_1^{-1} g_1' \in U(n)$.  It follows that the $U(n)$ action acts transitively on vectors in $U(n+1)$ with the same first column.  In particular, since $\mathcal{F}$ contains a point with the same first column as $g_1$, the orbit through $g_1$ contains a point in $\mathcal{F}$.

We finally argue that no two distinct points in $\mathcal{F}$ are in the same orbit.  Observe that an element $(g_1,I)$ of $\mathcal{F}$ is determined by the lengths of $(g_1)_{11}$, $(g_1)_{21}$, and $( (g_1)_{31},..., (g_1)_{n+1,1})$.  Thus, it is sufficient to show that these three lengths are invariant under the $J$-action.

To that end, simply observe that for any $(g_1,I)\in \mathcal{F}$, the element $$g_1':=\diag(w_2z^{q_1}, w_3 z^{q_2}, C) g_1 \diag( z^{p} w_1,  D)^{-1}$$ has first column given by $$( w_2z^{q_1} w_1^{-1} z^{-p_1} (g_1)_{11}), w_3z^{q_2} w_1^{-1} z^{-p_1} (g_1)_{21}, C ( (g_1)_{31},...,(g_1)_{n+1,1})^T.$$  Thus, $|(g_1')_{11}| = |(g_1)_{11}|$, $|(g_1')_{21}| = |(g_1)_{21}|$, and, since $C\in U(n-1)$, $$|((g_1')_{31},..., (g_1')_{n+1,1})| = |((g_1)_{31},..., (g_1)_{n+1,1})|,$$ as claimed.

\end{proof}

Let $\pi:G\times G\rightarrow \mathcal{E}_{p,q_1,q_2}$ denote the canonical projection map.  To determine the curvature of $\mathcal{E}_{p,q_1,q_2}$, we will restrict attention to points in $\pi(\mathcal{F})$.  This is justified by the following corollary.

\begin{corollary}\label{cor:checkF} If the set of points in $\mathcal{F}$ of positive curvature is open and dense in $\mathcal{F}$, then $\mathcal{E}_{p,q_1,q_2}$ is almost positively curved.  If there is a non-empty open set of points of $\mathcal{F}$ having at least one zero-curvature plane, then $\mathcal{E}_{p,q_1,q_2}$ is not almost positively curved.

\end{corollary}

\begin{proof}  It is enough to establish the following two claims.  First, if $U\subseteq \mathcal{F}$ is open, the the orbit $(U(n)\times T^3)\ast \pi(U)$ is open in $\mathcal{E}_{p,q_1,q_2}$.  Second, if $U$ is dense in $\mathcal{F}$ then $(U(n)\times T^3)\ast \pi(U)$ is dense in $\mathcal{E}_{p,q_1,q_2}$.

We begin with the first claim.  Since $\pi$ is a submersion, it is sufficient to show that the orbit $X:=(U(n)\times T^3) \ast (G\times L)\cdot U$ is open in $G\times G$. We will, in fact, show that the complement $X^c$ of this orbit is closed.  So, select a sequence of points $[((g_1)_i,(g_2)_i)]$ in $X^c$ and suppose this sequence converges to $[(g_1,g_2)]$.  We need to show that $[(g_1,g_2)]\in X^c$ as well.

To that end, let $g_i:= (g_2)_i^{-1} (g_1)_i$, $g:= g_2^{-1} g_2$, and let $(h_i,I), (h,I)\in \mathcal{F}$ denote the unique elements of $\mathcal{F}$ which are orbit equivalent to $((g_1)_i, (g_2)_i)$ and $(g_1,g_2)$.  Since $((g_1)_i, (g_2)_i)$ converges to $(g_1,g_2)$, it follows that the lengths of the entries $|(g_i)_{11}|$ converge to $|(g)_{11}|$, and that similarly the lengths $|(g_i)_{21}|\rightarrow |(g)_{21}|$ and $|( (g_i)_{31},..., (g_i)_{n+1,1})|\rightarrow |( (g)_{31},...,(g)_{n+1,1})|$.  From the proof of Proposition \ref{prop:cohom2}, these lengths determine the corresponding points in $\mathcal{F}$.  It follows, then, that $h_i$ converges to $h$.  As $h_i\in U^c$ for all $i$ and $U^c$ is closed, $h\in U^c$, which then implies that $[(g_1,g_2)]\in X^c$.

We now prove denseness.  So, suppose $U\subseteq \mathcal{F}$ is dense and let $V\subseteq \mathcal{E}_{p,q_1,q_2}$ be any non-empty open set.  We need to show that $\pi((U(n)\times T^3)\ast U)\cap V\neq \emptyset$.  Since $V\neq \emptyset$, we may select $[(g_1,g_2)]\in V$.

Now, from the proof of Proposition \ref{prop:cohom2}, there is a diffeomorphism $\rho:\mathcal{E}_{p,q_1,q_2}\rightarrow \mathcal{E}_{p,q_1,q_2}$ induced  by multiplication by elements in $U(n)\times T^3$,  for which $\rho([(g_1,g_2)]\in \pi(\mathcal{F})$.    Then $\rho(V)$ is an open set intersecting $\pi(\mathcal{F})$.  It follows that $\pi^{-1}(\rho(V))$ is an open set intersecting $\mathcal{F}$.  By definition of denseness, there is an element $(u,I)\in U\cap \pi^{-1}(\rho(V))$.  Then $\pi(u,I) \in \pi(U)\cap \rho(V)$ so $\rho^{-1}\pi(u,I)\in \pi( (U(n)\times T^3)\ast U)\cap V$.

\end{proof}

We conclude this section with an adaptation of Proposition \ref{prop:pregiveseqns} which is valid for any $A(t,r)\in \mathcal{F}.$

\begin{proposition}\label{prop:giveseqns} Suppose $X$ and $Y$ fall into type (iv) or type (v) of Proposition \ref{prop:martin5.1}.  Then the left translations of $\widehat{X}$ and $\widehat{Y}$ to $(\diag(A(t,r), I_{n-2}, I_{n+1})\in G\times G$ span a horizontal zero-curvature plane if and only if all of the following conditions are satisfied, where $\epsilon = 1$ for type (iv) and $\epsilon = 0$ for type (v).   \begin{itemize}\item \begin{multline}\label{eqn:4.2}
    \lambda_1 \alpha p \cos^2t\cos^2r - \lambda_1 \alpha q_1 \\+ 2p (\epsilon \im(x_2)\cos^2t\sin r + \im(x_3)\cos t \sin t )\cos r  = 0
\end{multline}

\item \begin{equation}\label{eqn:4.1}
    p\cos^2t(\beta \sin^2r + \epsilon\cos^2r) + 2p\im(y_3)\cos t \sin t \sin r -\beta q_2-\epsilon q_1 = 0
\end{equation}

\item  For $$V = \begin{bmatrix}( \epsilon x_2 \cos^2 r + \epsilon \overline{x_2} \sin ^2 r - i\lambda_1\alpha \cos r \sin r ) \cos t + \overline{x_3} \sin r \sin t \\ \cos r [x_3 \cos^2 t + (-i\lambda_1\alpha \cos r -2 \epsilon \im(x_2) i\sin r ) \sin t \cos t +\overline{x_3} \sin ^2 t]  \\ x_j \cos t \cos r, j\in \{4,...,n+1\}\end{bmatrix}$$ and $$W = \begin{bmatrix}     \cos r ((\beta - \epsilon)i \sin r \cos t - \overline{y_3} \sin t)\\    y_3 \cos ^2t \sin r - i(\beta \sin^2r +\epsilon \cos^2r)\sin t \cos t +\overline{y_3} \sin^2 t \sin r \\ y_j \cos t \sin r , j\in \{4,...,n+1\}\end{bmatrix},$$ $\{V,W\}$ is  linearly dependent. 

\end{itemize}

\end{proposition}

\begin{proof}The first two bullet points are simplified versions of equations 1 and 2 of Proposition \ref{prop:pregiveseqns}.  For the last bullet point, a simple calculation shows that $V$ and $W$ are the $(2,j)$ entries of the matrices $(Ad_{B^{-1}} \phi_1(X))$ and $Ad_{B^{-1}}Y$, respectively.  Since the $\mathfrak{p}$-component is determined by these entries, condition 3 of Proposition \ref{prop:pregiveseqns} implies that $\{V,W\}$ is linearly dependent.

\end{proof}

%%%%%%%%%%%%%%%%%%%%%%%%%%%%%%%%%%%%%
\section{Curvature on generalized Eschenburg spaces}\label{sec:easycurv}
%%%%%%%%%%%%%%%%%%%%%%%%%%%%%%%%%%%%%

In this section, we begin the analysis of the curvature of generalized Eschenburg spaces of cohomogeneity two.  The main result is the following:

\begin{proposition}\label{prop:easysummary} Suppose $(p,q_1,q_2)$ is admissible.  If $$(p,q_1,q_2)\notin \{(0,0,1),(0,1,0),  (1,1,0)\},$$ then $\mathcal{E}_{p,q_1,q_2}$ only has a nowhere dense set of points containing zero-curvature planes of the following types from Proposition \ref{prop:martin5.1}:

\begin{itemize}\item type (i), (ii), (iii), (v)

\item type (iv) with at least one $y_j \neq 0$ for some $j\geq 4$

\item type (iv) with either $V=0$ or $W=0$, which implies $y_j = 0$  for all $j\geq 4$. 

\end{itemize}  On the other hand, $\mathcal{E}_{0,0,1}$ has a non-empty open set of of points having zero-curvature planes of type (i), $\mathcal{E}_{0,1,0}$ has a non-empty open set of points having zero-curvature planes of type (ii), and $\mathcal{E}_{1,1,0}$ has a non-empty open set of points having zero-curvature planes of type (iv).
\end{proposition}

Our general approach is to work at a point $A(t,r)\in \mathcal{F}$ and assume there is a zero-curvature plane.  Via Proposition \ref{prop:giveseqns}, this will give a set of equations in the variables $(p,q_1,q_2), (t,r)$, and the entries of matrices $X$ and $Y$ which must be satisfied.   Working with these equations, we will eventually determine a finite set of polynomial functions in $p,q_1,q_2,\cos r,\cos t,\sin r,\sin t$ with the property that at least one of them must solved by $(p,q_1,q_2,t,r)$.  For each fixed $(p,q_1,q_2)$, the zero sets of these polynomials determine Zariski closed subsets of $\mathcal{F}$.  As long as each zero set is proper, that is, as long as each polynomial is not identically zero, the union of these zero sets is nowhere dense.  

%With Propositions \ref{prop:martin5.1} and \ref{prop:giveseqns}, we have a method of finding all $A\in \mathcal{F}$ with zero-curvature planes. We will show that such points do not occur on any open and dense set, beginning with types (i), (ii), and (iii), then  type (v), and finally type (iv), which requires significantly more analysis than the others. 

Towards proving Proposition \ref{prop:easysummary}, we begin with types (i), (ii), and (iii).

\begin{proposition}\label{prop:first3types}
    Suppose $(p, q_1, q_2) \not\in \{(0, 1, 0), (0, 0, 1)\}$ is admissible.  Then, on an open and dense set, $\mathcal{E}_{p,q_1,q_2}$ has no zero-curvature planes of type (i), (ii), or (iii).   On the other hand, $\mathcal{E}_{0,0,1}$ has type (i) zero-curvature planes on an open and dense set, while $\mathcal{E}_{0,1,0}$ has type (ii) zero-curvature planes on an open and dense set.
\end{proposition}

\begin{proof}
    We begin by proving the first statement. Suppose $\mathcal{E}_{p,q_1,q_2}$ has a zero-curvature plane of type (i), (ii), or (iii) at some point $A(t,r)\in \mathcal{F}$.  We apply equation 2 of Proposition \ref{prop:pregiveseqns} to find that for $Y$ in type (i), (ii) or (iii),
    \begin{itemize}
        \item[i)]$p\cos^2 r\cos^2 t-q_1 = 0$, 
        \item[ii)]$p\cos^2 t\sin^2 r-q_2 = 0$, or 
        \item[iii)]$p\cos^2 t-q_1-q_2 = 0$.
    \end{itemize}

    The proposition now follows for any $(p,q_1,q_2)$ for which none of these equations is an identity, so we now consider the possibility that at least one is not an identity.

    First, if either of the first two is an identity, then we find that $p=0$ and one of $q_1$ or $q_2$ vanishes.  Admissibility then implies that $(p,q_1,q_2)\in \{(0,1,0), (0,0,1)\}$, which is a contradiction.  On the other hand, if equation (iii) is an identity, then admissibility implies $(p,q_1,q_2) = (0,-1,1)$.  However, in \cite{Wi}, Wilking showed that his metric on $\mathcal{E}_{0,-1,1}$ is almost positively curved.

    \bigskip

 We now turn attention to $\mathcal{E}_{0,0,1}$.  We use the type (i) vectors of Proposition \ref{prop:martin5.1} with $Y = \diag(i,0,...,0)$ and where $X = (x)_{ij}$ is the matrix whose only non-zero entries are $x_{23} = -\overline{x}_{32} = -i\sin r\cot t$ and $x_{33} = i\frac{(2-\cos^2 r)\cos^2 t-1}{\sin^2 t}$.  Then it is straightforward to verify that $X$ and $Y$ satisfy all the conditions of Proposition \ref{prop:Kcurvature}.

Finally, for $\mathcal{E}_{0,1,0}$, we argue analogously using type (ii) vectors.  Here, $Y = \diag(0,i,0,...,0)$ and $X=(x)_{ij}$ with the non-zero entries $x_{13} = -\overline{x}_{13} = -i\cos r\cot t$ and $x_{33} = -i$.

\end{proof}

%\begin{note}
    %Wilking proved that $(0, -1, 1)$ has almost-positive curvature CITE. We will show that both $(0, 1, 0)$ and $(0, 0, 1)$ have zero-curvature planes of other types on open and dense sets (see Propositions \ref{prop:casevmain}, \ref{prop:caseivsubcase1}, and \ref{prop:caseivsubcase2}). With all three exceptions covered elsewhere, no further analysis is needed on cases (i), (ii), and (iii). 
%\end{note}

We next turn to type (v).  Recall that from Proposition \ref{prop:giveseqns} that if there a zero-curvature plane of type (v), then $\{V,W\}$ must be linearly dependent.  We analyze the possibilities that $V= 0$ or $W = 0$ or $V = sW$ for some $s\in \mathbb{R}\setminus\{0\}$ separately, beginning with the first two.

\begin{proposition} For any admissible $(p,q_1,q_2)$, zero-curvature planes of type (v) with either $V=0$ or $W=0$ occur only on a nowhere dense set.
\end{proposition}

\begin{proof}

%In case (v), we use (insert ref here to earlier in paper - the fact that smth and smth are linearly dependent) together with (-retr?) and the conditions on case (v) to form a system of equations, which we will show cannot be solved. First we will handle the case where adx = 0 or ady = 0, and then the case where both vectors are nonzero. 

Suppose there is such a plane at a point $A(t,r)\in \mathcal{F}$.  Assume first that $V = 0$.  Then the last entry of $V$ yields $x_j\cos r\cos t = 0$, so that either $\cos r\cos t = 0$ or $x_j =0$ for all $j\geq 4$ on an open and dense set.  Excluding the nowhere dense set defined by $\cos r\cos t = 0$, we therefore may assume $x_j =0$ for all $j\geq 4$.

If we further exclude the nowhere dense subsets where $\sin t =0$ and where $\cos^2 t-\sin^2 t = 0$, we solve the first two entries of the equation $V = 0$ for $x_3$, finding

    \begin{equation}
        x_3 = \frac{i\lambda_1\alpha\cos r\cos t\sin t}{\cos^2 t-\sin^2 t}
    \end{equation}
    and
    \begin{equation}
        x_3 = \frac{i\lambda_1\alpha\cos r\cos t}{\sin t}.
    \end{equation}
    Equating these, we find that either $\lambda_1 \alpha = 0$, or a non-identity rational equation in $(t,r)$ must be satisfied.  If $\lambda_1 \alpha = 0$, then $\alpha = x_3 = 0$, thus $X = 0$, which contradicts the fact that $\{X,Y\}$ are linearly independent.  On the other hand, clearing denominators in the rational equation yields a non-identity polynomial equation which must be satisfied by $(t,r)$.

\bigskip  We now assume $W=0$.  As before, this implies $y_j = 0$ for all $j\geq 4$ off of a nowhere dense subset of $\mathcal{F}$.  Off of a nowhere dense set, we may solve the first two entries of the equation $W=0$ for $y_3$, finding 
 
%%%%%%%%%%%%%%%%%%%% NOTE
% originally we solved to find cos^2(t) = 0 is the only solution, but that makes x3 = 0. cosr=0 gives the same solution and blah blah, so can we just say that x3 = 0 is the only solution? we are going to rule out all trig = 0 anyway, which we should probably do once at the top of all this instead of in each proposition..? 

  %  Case 2: ady = 0. Then we have 
   % \begin{equation}\label{eqn:5.3rhs}
    %    i  \beta \cos\left(r\right) \cos\left(t\right) \sin\left(r\right) - \overline{y_{3}} \cos\left(r\right) \sin\left(t\right) = 0,
    %\end{equation}
    %\begin{equation}\label{eqn:5.4rhs}
     %   -i \beta \cos\left(t\right) \sin^{2}\left(r\right) \sin\left(t\right) + y_{3} \cos^{2}\left(t\right) \sin\left(r\right) + \overline{y_{3}} \sin\left(r\right) \sin^{2}\left(t\right) = 0,,
    %\end{equation}
    %\begin{equation}\label{eqn:5.5rhs}
     %   y_{j} \cos\left(t\right) \sin\left(r\right)= 0, j \in \{4, 5,...,n+1\}.
    %\end{equation}
    %Taking an analogous approach, we see that $y_j = 0$ for $j \in \{4,...,n+1\}$ and we have two expressions for $y_3$ from equations \ref{eqn:5.3rhs} and \ref{eqn:5.4rhs}:
    \begin{equation}
        y_3 = \frac{i\beta\sin r\cos t\sin t}{\cos^2 t-\sin^2 t}
    \end{equation}
    and
    \begin{equation}
        y_3 = \frac{i\beta\sin r\cos t}{\sin t}.
    \end{equation}

    As in the proof where $V = 0$, setting these equal to each other either yields the contradiction that $Y = 0$ or leads to a polynomial equation which must be satisfied by $(t,r)$.

\end{proof}

We now handle type (v) planes with $V = sW$ for a non-zero real number $s$.

%%%%%%%%%%%%%%%%%%%%%%%%%%%%%%%%%%%%%%%%%%%%%
\begin{proposition}\label{prop:casevmain}
    If $(p,q_1,q_2)$ is admissible and is not one of $(0,0,1)$ or $(0,1,0)$, then there is a zero curvature plane of type (v) with $V = sW$ for some non-zero $s \in \mathbb{R}$ only on a nowhere dense subset of points in $\mathcal{E}_{p,q_1,q_2}$.
\end{proposition}

\begin{proof}
%%%%%%%%%%%%%%%NOTE: 
% Write something about how there is a scaling factor but we can pretend there isnt... and ofc where the equations come from but ill have to edit this whole thing to add that in so it's ok for now

      We can rescale $X$ so that $V = W$, i.e., so that $s=1$.  The equation $V=W$ is equivalent to equations \eqref{eqn:5.3}, \eqref{eqn:5.4}, and \eqref{eqn:5.5} below.

    \begin{multline}\label{eqn:5.3}
        -i  \lambda_1 \alpha \cos\left(r\right) \cos\left(t\right) \sin\left(r\right) + \overline{x_{3}} \sin\left(r\right) \sin\left(t\right) \\
        = i  \beta \cos\left(r\right) \cos\left(t\right) \sin\left(r\right) - \overline{y_{3}} \cos\left(r\right) \sin\left(t\right)
    \end{multline}
    \begin{multline}\label{eqn:5.4}
        \lambda_1 \alpha \cos^{2}\left(r\right) \cos\left(t\right) \sin\left(t\right) + x_{3} \cos\left(r\right) \cos^{2}\left(t\right) + \overline{x_{3}} \cos\left(r\right) \sin^{2}\left(t\right) \\
        = -i \beta \cos\left(t\right) \sin^{2}\left(r\right) \sin\left(t\right) + y_{3} \cos^{2}\left(t\right) \sin\left(r\right) + \overline{y_{3}} \sin\left(r\right) \sin^{2}\left(t\right)
    \end{multline}
    \begin{equation}\label{eqn:5.5}
        x_{j} \cos\left(r\right) \cos\left(t\right) = y_{j} \cos\left(t\right) \sin\left(r\right), j \in \{4, 5,...,n+1\}
    \end{equation}

    Taking the real parts of \ref{eqn:5.3} and \ref{eqn:5.4} separately, we get the equations
    \begin{equation}\label{eqn:v.re3}
        \re (x_3) \sin\left(r\right) \sin\left(t\right) =-\re (y_3) \cos\left(r\right) \sin\left(t\right)
    \end{equation}
    \begin{equation}\label{eqn:v.re4}
        \re (x_3) \cos\left(r\right) =\re (y_3)\sin\left(r\right)
    \end{equation}

 Equations \ref{eqn:v.re3} and \ref{eqn:v.re4} tell us that $\re(x_3) = \re(y_3) = 0$ off of a nowhere dense set.   Now, the imaginary parts of \ref{eqn:5.3} and \ref{eqn:5.4} are

    \begin{multline}\label{eqn:v.im3}
        \lambda_1 \alpha \cos r \cos t \sin r - \im (x_3) \sin r \sin t \\
        = \beta \cos r \cos t \sin r + \im (y_3) \cos r \sin t
    \end{multline}
    and 
    \begin{multline}\label{eqn:v.im4}
        \lambda_1 \alpha \cos^{2}r \cos t \sin t + \im (x_3) \cos r \cos^{2} t - \im (x_3) \cos r \sin^{2} t \\
        = - \beta \cos t \sin^{2} r \sin t +\im (y_{3}) \cos^{2} t \sin r -\im (y_3) \sin r \sin^{2} t
    \end{multline}
%%%%%%%%%%%%%%%%%%%%%%%%%%%%%%%%%%%%%%%%%%%%%%%%%%%%%%%%%%%%%%%%%%%%%%%%%%%%%%%%%%%%%%%%%%%%%%%%%%%%%%%%%%%%%%%%%%%%%%%%%%%%%%%%%%%%%%%%%%%%%%%%%%%%%%%%%%%%%%%%%%%%%%%%%%%%%%%%%%

    We now observe that \ref{eqn:4.1}, \ref{eqn:4.2}, \ref{eqn:v.im3}, and \ref{eqn:v.im4} form a homogeneous linear system of four equations in the four variables $\lambda_1\alpha, \beta, \im(x_3),$ and $\im(y_3)$.   We let $d$ denote the determinant of the matrix of coefficients of this system.  It is given by the expression
    \begin{multline}
        2  (p^{2} \sin^5 r - p^{2} \sin^3 r) \cos^6 t -2(p^{2} - p q_{1}) \sin^3 r - (p^{2} - 2  p q_{1}) \sin r) 
        \cos^4 t\\
        + 2  p q_{1} \cos^2 t \sin r - (2  (p^{2} \sin^6 r - p^{2} \sin^4 r) \cos^6 t - ((p^{2} + 2 p q_{2}) \sin^4 r \\
        - (p^{2} - 2  p q_{1} + 2 p q_{2}) 
        \sin^2 r) \cos^4 t\\ + ((p q_{1} + p q_{2}) \sin^2 r - (p - 2  q_{1}) q_{2}) \cos^2 t - q_{1} q_{2}) \sin t.
    \end{multline}

    We claim that $d$ is not identically zero.  To see this, observe that as powers of $\cos t$ are linearly independent, it is sufficient to verify that the coefficient of $\cos^6 t$ is non-zero.  This coefficient is $$p^2(-\sin^6 r + 2\sin^5 r + \sin^4 r - 2\sin^3 r),$$ which is non-zero so long as $p\neq 0$.  On the other hand, if $p=0$, then $d = q_1q_2(2\cos^2 t - 1)\sin t$, which is not identically zero unless $q_1 q_2 = 0$.  But then admissibility implies $(p,q_1,q_2) \in \{(0,0,1), (0,1,0)\}$, giving a contradiction.

    We next claim that if $d\neq 0$, then there are  zero-curvature planes of type (v) with $V=W$ only on a nowhere dense set.  Indeed, if $d\neq 0$, the corresponding linear system has a unique solution $(\lambda_1\alpha, \beta, \im(x_3), \im(y_3)) = (0,0,0,0)$.

    From equation \ref{eqn:5.5}, $y_j = x_j\cot r$ for each $j\in \{4,...,n+1\}$. Since $x_3 = y_3 = 0$, the constraint $\sum_{j=3}^{n+1}x_j\overline{y_j} = 0$ of Proposition \ref{prop:martin5.1} can be rewritten 
    $$\cot r \sum_{j=4}^{n+1} \left(\re(x_j)^2 + \im(x_j)^2\right) = 0,$$ 
    which cannot be satisfied unless $(0,x_3,...,x_{n+1})^T = \vec{0}$ %\footnote{vec 0 is correct? needs a ^T or no? {\color{blue} I agree it is confusing.  I think there are different conventions about whether you should think of $\mathbb{R}^n$ as a bunch of row vectors of column vectors.  One fix that avoids this is writing ``which cannot be satisfied unless $(0,x_3,...,x_{n+1})^T$ vanishes." or something like that.  Whatever you want to do sounds good to me.}}
    or $\cot r =0$.  Thus, off of the nowhere dense set where $\cot r = 0$, this implies $X = 0$, giving a contradiction.  Thus, the union of the solution set of the equations $d=0$ and $\cot r=0$ defines a nowhere dense subset of $\mathcal{F}$ containing all the points at which zero-curvature planes of type (v) with $V = sW$ can occur.

\end{proof}

We have now verified that apart from a few exceptional $(p,q_1,q_2)$, zero-curvature planes of types (i), (ii), (iii), and (v) of Proposition \ref{prop:martin5.1} can only occur on nowhere dense sets.  All that remains is verifying the same for type (iv). We will break this up into four cases, with the last taking considerably more effort to solve.

The four cases are 
\begin{enumerate}
    \item $V=0$
    
    \item  $W=0$
    
    \item  $V= sW$ for some non-zero $s\in \mathbb{R}$ and $y_j \neq 0$ for at least one $j = 4,..., n+1$, and
    \item  $V= sW$ for some non-zero $s\in \mathbb{R}$ and $y_j =0$ for all $j = 4,...,n+1$.
\end{enumerate}

We recall that in addition to all of the conditions of Proposition \ref{prop:giveseqns}, type (iv) zero-curvature also have the conditions that 

\begin{equation}\label{eqn:4.7}
    x_2\not=0,
\end{equation}
\begin{equation}\label{eqn:4.8}
    \beta = 1 - \sum_{j=3}^{n+1} |y_j|^2,
\end{equation} and 
\begin{equation}\label{eqn:4.9}
    x_j=-ix_2y_j, j\geq 3.
\end{equation}

We now handle the first case.

%We also include a proof, Proposition \ref{prop:typeivfirstproof} that establishes a portion of our main result for all type (iv) planes. Though this proof is more general, it is weaker than the final result we reach by breaking type (iv) planes into the above subcases. 

% \begin{equation}\label{eqn:4.3}
%     ( x_2 \cos^2 r + \overline{x_2} \sin ^2 r - i\lambda\alpha \cos r \sin r ) \cos t + \overline{x_3} \sin r \sin t = s \cos r ((\beta - 1)i \sin r \cos t - \overline{y_3} \sin t)
% \end{equation}
% \begin{multline}\label{eqn:4.4}
%     \cos r [x_3 \cos^2 t + (-i\lambda\alpha \cos r + (\overline{x_2}-x_2)\sin r ) \sin t \cos t +\overline{x_3} \sin ^2 t] \\
%     = s[y_3 \cos ^2t \sin r - i(\beta \sin^2r +\cos^2r)\sin t \cos t +\overline{y_3} \sin^2 t \sin r]
% \end{multline}
% \begin{equation}\label{eqn:4.5}
%     x_4 \cos t \cos r = s y_4 \cos t \sin r
% \end{equation}

\begin{proposition} \label{prop:caseivsubcase1.1} Suppose $(p, q_1, q_2) \neq (0, 0, 1)$.  Then there is a zero-curvature plane of type (iv) with $V=0$ only on a nowhere dense subset of points.

\end{proposition}

\begin{proof}  First, by inspecting the last entry of the equation $V = 0$, we see that off of a nowhere dense set, we must have $x_j = 0$ for all $j\geq 4$.

    Considering the real and imaginary parts of the first two entries of the equation $V=0$ gives the following four equations: 
    \begin{equation}\label{eqn:iv.re3lhs}
        \re( x_2)\cos t + \re (x_3) = 0,
    \end{equation}
    \begin{equation}\label{eqn:iv.im3lhs}
        (\im(x_2) (\cos^2 r - \sin^2 r) - \lambda_1\alpha\cos r\sin r)\cos t - \im (x_3) \sin r\sin t = ,
    \end{equation}
    \begin{equation}\label{eqn:iv.re4lhs}
        \re (x_3) \cos r = 0,
    \end{equation}
    \begin{equation}\label{eqn:iv.im4lhs}
        -\cos r[(\lambda_1\alpha\cos r + 2\im(x_2) \sin r)\sin t\cos t - \im (x_3) (\cos^2 t - \sin^2 t)] = 0.
    \end{equation}
    
    From equations \ref{eqn:iv.re3lhs} and \ref{eqn:iv.re4lhs}, we find $\re(x_2) = \re(x_3) = 0$ off of a nowhere dense set.
    
    Now, we consider the system of three homogeneous linear equations \ref{eqn:iv.im3lhs}, \ref{eqn:iv.im4lhs}, and \ref{eqn:4.2} in the three unknowns $\lambda_1\alpha, \im(x_2), \im(x_3)$.  Let $d$ denote the determinant of the coefficient matrix, which has the expression

    \begin{equation}
        d = (p-2q_1)\cos^2r\cos^3t + q_1\cos^3t.
    \end{equation}

    If $d\neq 0$, then the unique solution to the system of three equations is $\lambda_1 \alpha = \im(x_2) = \im(x_3) = 0$, which then forces $X= 0$, giving a contradiction.  Thus, we must have $d=0$, giving a polynomial which must be satisfied.  Thus, the proposition is proven as long as $d = 0$ is not an identity.  The determinant is only identically $0$ when $p=q_1 = 0$.  Admissibility then implies $(p, q_1, q_2) = (0, 0, 1)$, giving a contradiction.
    
\end{proof}

We now turn to the second case, where $W=0$.

%%%%%%%%%%%%%%%%%%%%%%%%%%%%%%%%%%%%%%
\begin{proposition} \label{prop:caseivsubcase1.2} Suppose $(p, q_1, q_2) \neq (1, 1, 0)$.  Then there is a zero-curvature plane of type (iv) with $W=0$ only on a nowhere dense subset of points.  On the other hand, when $(p,q_1,q_2) = (1,1,0)$, there is a non-empty open set of points admitting such a zero-curvature plane.
\end{proposition}

\begin{proof}  We begin with the first claim. Assume $(p,q_1,q_2)$ is admissible and $(p,q_1,q_2)\neq (1,1,0)$.  By inspection, the last entry of the equation $W = 0$ implies that off of a nowhere dense set, $y_j = 0$ for all $j\geq 4$.  Now, taking the real and imaginary parts of the first two entries of the equation $W=0$, we obtain the following equations.  
    \begin{equation}\label{eqn:iv.re3rhs}
        -\re (y_3) \sin t\cos r = 0,
    \end{equation}
    \begin{equation}\label{eqn:iv.im3rhs}
        \cos r((\beta - 1)\sin r\cos t + \im (y_3)) = 0,
    \end{equation}
    \begin{equation}\label{eqn:iv.re4rhs}
        \re (y_3) \sin r = 0,
    \end{equation}
    \begin{equation}\label{eqn:iv.im4rhs}
        \im (y_3) \sin r(\cos^2 t-\sin^2 t) - (\beta \sin^2 r + \cos^2 r)\sin t\cos t = 0.
    \end{equation}
    
    From \ref{eqn:iv.re3rhs} and \ref{eqn:iv.re4rhs}, we see that we must have $\re(y_3) = 0$ off of a nowhere dense set.
    Then from \ref{eqn:iv.im3rhs} and \ref{eqn:iv.im4rhs}, we find
    $$\im(y_3) = \frac{\sin t}{\cos t \sin r}\text{ and } \beta = 1-\frac{\sin^2 t}{\cos^2 t \sin^2 r}.$$
    
    Substituting these values in \ref{eqn:4.1}, we have 
    \begin{equation} \label{eqn:4.1withsubs}
        \frac{{\left({\left(p - q_{1} - q_{2}\right)} \sin^{2}r - q_{2}\right)} \cos^{2}t) + q_{2}}{\cos^{2}t \sin^{2}r} = 0.
    \end{equation}
    The proof of the first claim is complete as long as the left-hand side of \ref{eqn:4.1withsubs} is not identically zero.  If it is identically zero, then we find that $p=q_1+q_2$ and $q_2 = 0$.  Admissibility then implies $(p,q_1,q_2) = (1,1,0)$, which is a contradiction.

\bigskip
    
    We now prove the second claim, so assume $(p,q_1,q_2) = (1,1,0)$.  If we define $x_j=y_j = 0$ for all $j\geq 4$, $y_3= \frac{\sin t}{\cos t\sin r}i$, $\beta = 1-\frac{\sin^2 t}{\cos^2 t\sin^2 r}$, $x_2 = 1$, $x_3 = -iy_3$, and we define $\lambda_1 \alpha$ by solving equation \ref{eqn:4.2}, then it is is to verify that Proposition \ref{prop:giveseqns} and equations \eqref{eqn:4.7}, \eqref{eqn:4.8}, and \eqref{eqn:4.9} are satisfied for any $(t,r)$.  Hence, there is a zero-curvature plane of type (iv) at every point.

\end{proof}

We now turn to the third case, where $V=sW$ with $s$ a non-zero real number and with $y_j \neq 0$ for some $j\geq 4$.

%%%%%%%%%%%%%%%%%%%%%%%%%%%%%%%%%%%%%%
\begin{proposition} \label{prop:caseivsubcase2} Suppose $(p,q_1,q_2)$ is admissible and not equal to $(0,0,1)$.  Then there can only be zero-curvature planes of type (iv) with $V=sW$ for a non-zero real number $s\neq 0$ and $y_j\neq 0$ for some $j\geq 4$ on a nowhere dense set.
\end{proposition}

\begin{proof}  After rescaling $X$, we may assume $s=1$.  We also assume without loss of generality that $y_4\neq 0$.  Then the equation $V=W$ yields the following equations.

    \begin{multline}\label{eqn:4.3}
        ( x_2 \cos^2 r + \overline{x_2} \sin ^2 r - i\lambda_1\alpha \cos r \sin r ) \cos t + \overline{x_3} \sin r \sin t =\\ \cos r ((\beta - 1)i \sin r \cos t - \overline{y_3} \sin t)
    \end{multline}
    \begin{multline}\label{eqn:4.4}
        \cos r [x_3 \cos^2 t + (-i\lambda_1\alpha \cos r + (\overline{x_2}-x_2)\sin r ) \sin t \cos t +\overline{x_3} \sin ^2 t] \\
        = y_3 \cos ^2t \sin r - i(\beta \sin^2r +\cos^2r)\sin t \cos t +\overline{y_3} \sin^2 t \sin r
    \end{multline}
    \begin{equation}\label{eqn:4.5}
        x_4 \cos t \cos r = y_4 \cos t \sin r
    \end{equation}

     Off of a nowhere dense set, equations \eqref{eqn:4.5} and  \ref{eqn:4.9} imply $x_2 = i\tan r$ and $x_3 = y_3\tan r$. Then we can substitute $x_3 = y_3\tan r$ into equations \ref{eqn:4.2}, \ref{eqn:4.3}, and \ref{eqn:4.4}. Taking the imaginary parts of equations \ref{eqn:4.3} and \ref{eqn:4.4} together with \ref{eqn:4.1} and \ref{eqn:4.2}, we have a system of four linear equations with the three unknown values $\lambda_1\alpha, \beta,$ $ \im(y_3)$. Let $d$ denote the determinant of the $4\times 4$ matrix obtained by augmenting the coefficient matrix of this system with an extra column of the values of the system. 
    The $d$ is given by $-\tan^2(r)$ multiplied by the following expression:

    \begin{multline}
        2 \cos^{5} t (p^{2} \cos r - p^{2}) (\sin^6 r -  \sin^4 r) + (4  (p q_{1} + p q_{2} - (p q_{1} + p q_{2}) \cos r) \sin^6 r\\ - (2  p q_{1} + 6  p q_{2} - (p q_{1} + 7  p q_{2}) \cos r) \sin^4 r - p q_{2} \cos r\\ + (2 p q_{2} + (p q_{1} - 2  p q_{2})\cos r) \sin^2 r)\cos t^{3} + (q_{1} q_{2} \cos r \\- (2  q_{1} q_{2} + (q_{1}^{2} + q_{1} q_{2}) \cos r)\sin^2 r) \cos t.
    \end{multline}
    
    In order for this over-determined system to have a solution, it is necessary and sufficient that $d=0$.  Thus, as long as $d$ is not identically zero, we have proven the proposition.

    The coefficient of $\cos^5 t$ is $2p^2(\cos r-1)(\sin^6 r-\sin^4 r)$ which is only identically zero if $p=0$.  Moreover, when $p=0$, we find $$d = -q_1\tan^2 r \cos t ( q_{2} \cos r - (2  q_{2} + (q_{1} +  q_{2}) \cos r)\sin^2 r), $$ which is not an identity unless $q_1 = 0$.  But if $p=q_1 = 0$, then admissibility implies $(p,q_1,q_2) = (0,0,1)$, which is a contradiction.

\end{proof}

The remaining case of type (iv) zero-curvature planes, where $V=W$ and $y_j = 0$ for all $j\geq 4$, is considerably more difficult than the remaining cases and will be postponed until later.  We are, however, ready to show that infinitely many $\mathcal{E}_{p,q_1,q_2}$ are almost positively curved.

\begin{proposition}\label{prop:newexamples} Suppose $(p,q_1,q_2)$ is admissible with both $q_1 + q_2< 0$ and $p\in [q_1+q_2,q_2]$.  Then Wilking's metric on $\mathcal{E}_{p,q_1,q_2}$ is almost positively curved.

\end{proposition}

\begin{note}There are infinitely many admissible triples $(p,q_1,q_2)$ which satisfy the hypothesis of Proposition \ref{prop:newexamples}.  For example, one can take $(p,q_1,q_2) = (2,1 -6k,3)$ for any $k\geq 1$.  In particular, Proposition \ref{prop:newexamples} yields infinitely many almost positively curved manifolds in each dimension of the form $4n-1$ with $n\geq 3$.

\end{note}

\begin{proof}  The assumptions in $p,q_1$, and $q_2$ imply that $$(p,q_1,q_2)\notin \{(0,0,1), (0,1,0), (1,1,0)\},$$ so Proposition \ref{prop:easysummary} implies that the only possible zero-curvature planes which can exist on a non-empty open set are those of type (iv) with $V=W$ and with $y_j = 0$ for all $j\geq 4$.

We will show that equations \eqref{eqn:4.1} and \eqref{eqn:4.8} are incompatible off of a nowhere dense set.  Solving \eqref{eqn:4.1} for $\beta$, we find $\beta = h(r,t) \im(y_3) + j(r,t)$ where $$h(r,t) = \frac{-2p \cos t \sin t \sin r}{p \cos^2 t \sin^2 r - q_2} \text{ and } j(r,t) = \frac{q_1 - p \cos^2 t \cos^2 r}{p \cos^2 t \sin^2 r - q_2}.$$  We observe that the denominator in these expressions is not identically zero:  if it were, then $p=q_2 = 0$, and admissibility implies $(p,q_1,q_2) = (0,1,0)$, yielding a contradiction.

Substituting this formula for $\beta$ into equation \eqref{eqn:4.8}, we find $$ h(r,t) \im(y_3) + j(r,t) = 1- \im(y_3)^2-\re(y_3)^2.$$  This equation is clearly equivalent to \begin{equation}\label{eqn:shownegative} \re(y_3)^2 = -\im(y_3)^2 + (-1-h(r,t))\im(y_3) + 1-j(r,t).\end{equation}

We will show that off of a nowhere dense set, the right hand side of \eqref{eqn:shownegative} is negative.  Since the left side is clearly non-negative, this will show that these kinds of type (iv) zero-curvature planes occur only on a nowhere dense set.

We view the right hand side as a quadratic polynomial in $\im(y_3)$.  The  leading coefficient is negative, so the right hand side is negative if and only if the discriminant $d(r,t)$ is negative 
% \footnote{but all following calculations conclude with $\leq 0$, not $ < 0$ {\color{blue}  We can clarify it if you want, but this is already addressed in the paragraph after (4.40):  Because $k$ is non-constant, if $k<=0$ for all $(r,t)$, then, since the set where $k=0$ is nowhere dense, $k<0$ off a nowhere dense set (the nowhere dense set given by $k = 0$.}}
off of a nowhere dense set.  Of course, $d(r,t)$ is negative on such a set if and only if $k(r,t):=d(r,t)(p\cos^2 t\sin^2 r-q_2)^2$ is as well.

A simple calculation reveals that \begin{equation}\label{eqn:check2}k(r,t) = 4p((p - q_1-
q_2)\sin^2 r - q_2)\cos^2 t + 4q_2(q_1 + q_2). \end{equation}

If $k(r,t)$ is a constant, then $q_2 =0$.  Admissibility then implies that $q_1 \geq 0$, which contradicts the fact that $q_1 + q_2 < 0$.  Thus, $k(r,t)$ is non-constant.  So, if we show that $k\leq 0$ for all $(r,t)$, it will follow that $k < 0$ off of a nowhere dense set, completing the proof.

We will prove that the maximum value of $k$, which must either occur on the boundary of $[0,\pi/2]\times [0,\pi/2]$ or at a interior critical point, is non-positive.  We begin by checking $k$ on the boundary.

We note that $k(0,t) = 4q_2( -p\cos^2 t +q_1 + q_2)$.  As $q_2 \geq 0$ and $q_1 + q_2 \leq p$, it follows that $k(0,t)\leq 0$ for all $t\in [0,\pi/2]$.

Similarly, $k(\pi/2,t)$ is equal to $$4p(p-q_1-2q_2)\cos^2 t  + 4q_2(q_1+q_2) = 4(q_2-p\cos^2 t)(q_1+q_2) + 4p(p-q_2)\cos^2 t.$$   The maximum of this expression occurs either when $t\in \{0,\pi/2\}$ or when its derivative vanishes.  But $k(\pi/2,0) = (p-q_2)(p-(q_1+q_2))\leq 0$ and $k(\pi/2,\pi/2) = q_2(q_1+q_2)\leq 0$.  The derivative of $k(\pi/2,t)$ is $-2p \cos t\sin t(p-q_1-2q_2)$.  This can only vanish for $t\in (0,\pi/2)$ if $p = 0$ or $p - q_1-2q_2 = 0$.  Both possibilities lead to $k(\pi/2,t) = q_2(q_1 + q_2)\leq 0$.

We also observe that $k(r,0)$ is equal to 
$$4p( (p-q_1-q_2)\sin^2 r - q_2) + 4 q_2(q_1 + q_2) = 4(-p\sin^2 r+q_2)(q_1+q_2) + 4p(p\sin^2 r - q_2),$$ 
which is a sum of non-positive terms and therefore is non-positive.

Finally, $k(r,\pi/2) = 4q_2(q_1+q_2)\leq 0$.  Thus, $k\leq 0$ on the boundary of $[0,\pi/2]\times [0,\pi/2]$.

We now investigate critical points in the interior of the rectangle.  We compute that the partial derivatives are given by $$k_r = 8p(p-q_1-q_2)\sin r\cos r\cos^2 t$$ and $$ k_t = -8p((p-q_1-q_2)\sin^2 r - q_2)\sin t\cos t. $$

Since both $r,t\in (0,\pi/2)$, there are no interior critical points unless $p=0$ or both $p=q_1 + q_2$ and $q_2 = 0$.  Both in both of these cases, we find $k(r,t) = 4q_2(q_1+q_2)\leq 0$.

\end{proof}

\section{The degree 4 polynomial}\label{sec:poly}

In the next two sections, we analyze the fourth and final case of type (iv), where $V = sW$ for some non-zero $s \in \mathbb{R}$ and $y_j=0$ for all $j \geq 4$. We will rely on the polynomial $f$, defined below, for our results in this case. Proposition \ref{prop:fpurpose} explains the importance of this polynomial. 

For each admissible $(p,q_1,q_2)$, we define the polynomial $f_{p,q_1,q_2}(x,y) := \sum_{i,j=0}^4 c_{i,j} x^i y^j$ where 
\begin{itemize}\item $c_{4,4} = -p^2 (q_1-q_2)^2$
\item $c_{4,j} = 0$ for $j=0,1,2,3$
\item $c_{3,4} = -2p^2(p-q_1 - 2q_2)(q_2-q_1)$
\item $c_{3,3} = -2p(q_1-q_2)(2p^2-5pq_1 - 3pq_2 + q_1^2+3q_1 q_2)$
\item  $c_{3,j} = 0$ for $j=0,1,2$
\item $c_{2,4} = -p^2(p-q_1-2q_2)^2ce$
\item $c_{2,3} = 2p(p-q_1-2q_2)(3pq_1 - p_1q_2 - 6q_1^2-2q_1q_2)$
\item $c_{2,2} =-p^2q_1^2 + 6p^2q_1q_2 - p^2q_2^2 - 4pq_1^3 - 4pq_1q_2^2 - q_1^4 - 6q_1^3q_2 - 9q_1^2q_2^2$
\item $c_{2,j} = 0$ for $j=0,1$
\item $c_{1,j} = 0$ for $j=0,4$
\item $c_{1,3} = -2 p q_1(p-q_1-2q_2)^2$
\item $c_{1,2} = 2q_1(2p + q_1)(q_1-q_2)(p-q_1-2q_2)$
\item $c_{1,1} = 2q_1(q_1+q_2)(pq_1 - pq_2 + q_1^2+q_1q_2)$
\item $c_{0,j}= 0$ for $j=3,4$.
\item $c_{0,2} = -q_1^2(p-q_1-2q_2)^2$
\item $c_{0,1} =-2q_1^2(q_1+q_2)(p-q_1-2q_2)$
\item $c_{0,0} = -q_1^2(q_1+q_2)$

\end{itemize}

With $f$ defined, we now explain its importance with the following result. 

\begin{proposition}\label{prop:fpurpose} Suppose $(p,q_1,q_2)$ is admissible and is not one of 
\begin{enumerate}\item $(0,1,0)$
\item $(0,0,1)$

\item $(1,1,0)$

\item $(0,-1,1)$.
\end{enumerate}  Then the Wilking metric on $\mathcal{E}_{\vec{p},\vec{q}}$ has almost positive curvature if and only if $f_{p,\vec{q}}(x,y) \leq 0$ for all $x,y \in [0,1]$.
\end{proposition}

\begin{proof} From Proposition \ref{prop:easysummary}, we know that under the hypothesis on $(p,q_1,q_2)$, this proposition is true if and only if the following kinds of zero-curvature planes only occur on a nowhere dense set:  type (iv) with $V=sW$ and $y_j = 0$ for all $j\geq 4$. Thus, for the duration of the proof, we restrict our attention to planes of this special form.

As in the proof of Proposition \ref{prop:caseivsubcase2}, we can rescale $X$ so that $s = 1$ and $V=W$. The last entry of the equation $V=W$ gives $x_j = 0$ for $j\geq 4$.  The first two entries give the following two equations.

\begin{multline}\label{eqn:eqn3} (x_2\cos^2 r + \overline{x}_2\sin^2 r -i \lambda_1\alpha\cos r\sin r)\cos t + \overline{x}_3\sin r \sin t \\= \cos r(i(\beta-1) \sin r \cos t - \overline{y}_3 \sin t)\end{multline}

\begin{multline} \label{eqn:eqn4} \cos r(x_3\cos^2 t + \overline{x}_3\sin^2 t + \sin t\cos t(-i\lambda_1 \alpha \cos r -2i \im (x_2)\sin r )) \\= y_3\cos^2 t \sin r +  \overline{y}_3\sin^2 t\sin r - i\sin t(\beta\sin^2 r+\cos^2 r)\cos t\end{multline}

We also recall that equations \eqref{eqn:4.8} and \eqref{eqn:4.9} must be satisfied.

To begin the analysis of these equations, we multiply both sides of \eqref{eqn:eqn3} by $\cos r\sin t$ and both sides of \eqref{eqn:eqn4} by $\sin r$ and subtract.  This gives \begin{multline}\label{eqn:step0} \cos t\cos r (-x_3\cos t\sin r + x_2\sin t)\\ = i\beta\cos t\sin t\sin r - y_3 \cos^2 t\sin^2 r- y_3\sin^2 t.\end{multline}

Using \eqref{eqn:4.9} to replace $x_3$ and then solving for $x_2$, we find \begin{equation}\label{eqn:step1} x_2 = - \frac{i y_3\cos^2 t\sin^2 r + i\overline{y}_3\sin^2 t + \beta\cos t\sin t \sin r }{\cos r\cos t(-y_3\cos t\sin r + i\sin t)}.\end{equation}

If $y_3 = 0$, then \eqref{eqn:4.8} implies $\beta = 1$, and $\eqref{eqn:4.1}$ reduces to $p\cos^2 t - q_2 -q_1 = 0$, which is not an identity unless $p=0$ and $q_1 = -q_2$.  Admissibility implies that $(q_1,q_2) = (-1,1)$, giving the triple $(p,q_1,q_2) = (0,-1,1)$.  Since, by hypothesis, $(p,q_1,q_2)\neq (0,-1,1)$, we conclude that $y_3\neq 0$.  Multiplying the right hand side of \eqref{eqn:step1} by $1  = \frac{y_3}{y_3}$ and using \eqref{eqn:4.8} to replace $|y_3|^2$ with $1-\beta$, we find 
$$x_2 = - \frac{iy_3^2\cos^2 t\sin^2 r + i\sin^2 t (1-\beta) +\beta y_3 \cos t\sin t \sin r}{y_3\cos r\cos t(-y_3\cos t\sin r + i\sin t)}.$$  
The numerator factors as $(-y_3\cos t\sin r + i\sin t)(-iy_3\cos t\sin r -(\beta-1)\sin t)$, which implies $$ x_2 = \frac{iy_3\cos t\sin r + (\beta-1)\sin t}{y_3\cos r\cos t}.$$

Noting that \eqref{eqn:4.8} implies $\frac{\beta - 1}{y_3} = -\overline{y}_3$, we obtain \begin{equation}\label{eqn:step2} x_2 = \frac{i\sin r\cos t - \overline{y}_3 \sin t}{\cos t\cos r}.\end{equation}

Substituting \eqref{eqn:step2} into \eqref{eqn:step0} and solving for $x_3$, we find \begin{equation}\label{eqn:step3} x_3 =  \frac{y_3\sin r\cos t + i(1-\beta)\sin t }{\cos t\cos r}.\end{equation}

Since $(p,q_1,q_2)\neq (0,0,1)$, we may substitute \eqref{eqn:step2} and \eqref{eqn:step3} into \eqref{eqn:4.2} and solve for $\lambda_1 \alpha$, finding 
\begin{equation}\label{eqn:step4} \lambda_1\alpha = \frac{2p(\beta\sin^2 t -1 + \cos^2 t\cos^2 r - 2\im( y_3) \sin r\sin t\cos t) }{p\cos^2 t\cos^2 r - q_1}.\end{equation}

Similarly, since $(p,q_1,q_2)\neq (0,1,0),$ we may solve \eqref{eqn:4.1} for $\beta$, finding \begin{equation}\label{eqn:step5} \beta = -\frac{-2p \im (y_3)\sin t\sin r\cos t +  p\cos^2 t\cos^2 r - q_1}{p\cos^2 t\sin^2 r  - q_2}.\end{equation}

We may now use \eqref{eqn:step5} in \eqref{eqn:step2},\eqref{eqn:step3}, and \eqref{eqn:step4} to express $x_2,x_3$, and $\lambda_1\alpha$ solely in terms of $y_3$.  As the expressions get quite long, we choose not to display them.  Defining $x_2,x_3,\lambda_1\alpha$, and $\beta$ by \eqref{eqn:step2},\eqref{eqn:step3},\eqref{eqn:step4}, and \eqref{eqn:step5}, it is easy to verify that for any $y_3\in \mathbb{C}$, \eqref{eqn:4.1}, \eqref{eqn:4.2}, and \eqref{eqn:4.9} are satisfied.  Moreover, by considering the difference of \eqref{eqn:eqn3} multiplied by $\cos r\sin t$ and \eqref{eqn:eqn4} multiplied by $\sin r$, it is also easy to verify that for any $y_3\in \mathbb{C}$, $\eqref{eqn:eqn3}$ is satisfied if and only if $\eqref{eqn:eqn4}$ is satisfied.

At this stage, we have established that there is a horizontal zero-curvature plane at $(\diag(A(r,s),I_{n-2}),I_{n+1})\in \mathcal{F}$ if and only if there is $y_3\in\mathbb{C}$ satisfying both \eqref{eqn:eqn3} and \eqref{eqn:4.8}.

Turning to \eqref{eqn:eqn3}, a straightforward calculation reveals that the real part of both sides of \eqref{eqn:eqn3} is $-\re (y_3)\sin t\cos r $.  In particular, this is satisfied for any value of $\re (y_3)$.  On the other hand, as long as $2p(p-2q_1 - q_2)\cos^2 r\cos^2 t + 2q_1(p - q_2)$ is not identically zero,  the imaginary part of \eqref{eqn:eqn3} can be solved for $\im (y_3)$, yielding a formula for $\im (y_3)$ whose numerator is 
\begin{multline} ((q_1 - q_2)\cos^2 r + p - 3q_1)p\cos^2 r \cos^4 t\\ + ((-p q_2 + q_1(q_1 + 3q_2))\cos^2 r + p q_1 \sin^2 r)\cos^2 t - q_1^2\sin^2 t - q_1q_2\end{multline} and whose denominator is \begin{equation}\label{eqn:denom}-2\sin r\sin t\cos t(p(p-2q_1 - q_2)\cos^2 r\cos^2 t + q_1(p - q_2)).\end{equation}

We observe that the list of excluded $(p,q_1,q_2)$ in the proposition statement guarantee that the denominator is not identically zero.  To see this, note that the denominator is identically zero if and only if $p(p-2q_1-q_2) = 0$ and $q_1(p-q_2) = 0$.  If $p=0$, the second equation forces $q_1=0$ or $q_2 = 0$, and admissibility then implies $\{q_1,q_2\} = \{0,1\}$, yielding $(p,q_1,q_2)\in \{(0,0,1), (0,1,0)\}$.  On the other hand, when $p = 2q_1 + q_2$ we see that both possibilities---$q_1 = 0$ or $p = q_2$---imply that $q_1 = 0$.  When $q_1 = 0$, admissibility implies $q_2 = 1$, so $(p,q_1,q_2) = (1,0,1)$.  But this is not admissible, as $\gcd(q_1,p-q_2) \neq 1$.

Having determined $\im (y_3)$ solely in terms of $(r,s)$, we observe that substituting this into \eqref{eqn:step5} yields 
$$\beta = \frac{p(p - 3q_1)\cos^2 t\cos^2 r + q_1(p + q_1)}{p(p - 2q_1 - q_2)\cos^2 r\cos^2 t + q_1(p - q_2)},$$ so the denominator of $\beta$ is a factor of the denominator of $\im (y_3)$.

At this stage, all that remains is to determine $\re (y_3)$ using the only remaining equation that   is not automatically solved: $\eqref{eqn:4.8}$.  We find that $\re (y_3) = \sqrt{1-\beta - \im (y_3)^2}$.  Of course, for this to be a valid definition of $\re (y_3)$, we must have $1-\beta - \im (y_3)^2 \geq 0$.   Expressing $1-\beta - \im (y_3)^2$ as a single fraction, the denominator is the square of \eqref{eqn:denom}, so it is automatically non-negative.  On the other hand, after expressing the numerator entirely in terms of $\cos^2 r$ and $\cos^2 t$, we set $x = \cos^2 r$ and $y = \cos^2 t$.  The numerator then is $f_{p,q_1,q_2}(x,y)$.

We have now shown that there is a zero-curvature plane of type (iv) at $[A(\sqrt{\arccos(y)},\sqrt{\arccos(x)}),I)]$ if and only if $f=f_{p,q_1,q_2}(x,y) \geq 0$.   If $f(x,y) > 0$ for some  $(x,y)\in (0,1)^2$, then by continuity, $f> 0$ in a neighborhood of $(x,y)$.  By Corollary \ref{cor:checkF}, we can then find an open set of points containing at least one zero-curvature plane.  On the other hand, if $f\leq 0$ on $(0,1)\times (0,1)$, then since $f$ is not identically zero, $f = 0$ only on a nowhere dense set.  Thus, $\mathcal{E}_{p,q_1,q_2}$ has almost positive curvature if and only if $f\leq 0$ on $[0,1]^2$.

\end{proof}

\section{\texorpdfstring{Analysis of $f_{p,q_1,q_2}$}{Analysis of f}}\label{sec:polyanalysis}

 In this section, we prove Theorem \ref{thm:main}.   For the convenience of the reader, we now indicate where each possibility for $(p,q_1,q_2)$ is handled. 

\iffalse \begin{tabular}{c|c|c}
$ (p,q_1,q_2)$ & Proposition & almost positive curvature?\\
\hline

$(0,1,0)$ & Proposition \ref{prop:first3types} & No \\
$(0,0,1)$ & Proposition \ref{prop:first3types} & No \\
$q_1+q_2 = 0 $ & Proposition \ref{prop:q1+q2=0} & iff $p\in \{0,1\}$\\
$q_1 + q_2 < 0$, $p\in [q_1 +q_2,q_2]$ & Proposition \ref{prop:newexamples} & Yes\\
$q_1 + q_2 < 0$, $p > q_2$ & Proposition \ref{prop:twocases}& No \\
$q_1 + q_2<0$, $p < q_1 + q_2$ & Proposition \ref{prop:DJ} & No\\ 
$q_1 + q_2 > 0$, $p < q_2$ & Proposition \ref{prop:twocases}& No\\
$q_1 + q_2 > 0$, $p = q_2$ & Proposition \ref{prop:hessian} &No\\
$q_1 + q_2 > 0$, $p>q_2$, $p < q_1 + q_2$ & Proposition \ref{prop:DJ} &No\\
$q_1 + q_2 > 0$, $p >q_2$, $p\geq q_1 + q_2$, $q_1 < 0$ & Proposition \ref{prop:DJ} & No\\
$q_1+q_2 > 0$, $p > q_2$, $p\geq q_1 + q_2$, $q_1 \geq 0$ & Proposition \ref{prop:Ralg} and \ref{prop:caseivsubcase1.2} & Yes, except $(p,q_1,q_2) = (1,1,0)$\\ 
\end{tabular}

\fi
\begin{center}
\begin{tabular}{c|c|c}
$ \mathbf{(p,q_1,q_2)}$ & \textbf{Proposition} & \textbf{a.p.c.?}\\
\hline
\hline

$(0,1,0)$ & Proposition \ref{prop:first3types} & No \\
\hline
$(0,0,1)$ & Proposition \ref{prop:first3types} & No \\
\hline
$q_1+q_2 = 0 $ & Proposition \ref{prop:q1+q2=0} & iff $p\in \{0,1\}$\\
\hline
$q_1 + q_2 < 0$,  & Proposition \ref{prop:newexamples} & Yes\\
$p\in [q_1 +q_2,q_2]$ & & \\
\hline
$q_1 + q_2 < 0$, $p > q_2$ & Proposition \ref{prop:twocases}& No \\
\hline
$q_1 + q_2<0$, $p < q_1 + q_2$ & Proposition \ref{prop:DJ} & No\\ 
\hline
$q_1 + q_2 > 0$, $p < q_2$ & Proposition \ref{prop:twocases}& No\\
\hline
$q_1 + q_2 > 0$, $p = q_2$ & Proposition \ref{prop:hessian} &No\\
\hline
$q_1 + q_2 > 0$, $p>q_2$,  & Proposition \ref{prop:DJ} &No\\
$p < q_1 + q_2$ & & \\
\hline
$q_1 + q_2 > 0$, $p >q_2$,  & Proposition \ref{prop:DJ} & No\\
$p\geq q_1 + q_2$, $q_1 < 0$ &  & \\
\hline
$q_1+q_2 > 0$, $p > q_2$,  & Proposition \ref{prop:Ralg} & Yes, \\ 
$p\geq q_1 + q_2$, $q_1 \geq 0$ & and \ref{prop:caseivsubcase1.2}  & except $(1,1,0)$ \\
\end{tabular}
\end{center}
\begin{proposition}\label{prop:q1+q2=0} Suppose $q_1+q_2 = 0$ so that $(q_1,q_2) = (-1,1)$.
Then the Wilking metric on $\mathcal{E}_{p,q_1,q_2}$ is almost positively curved if and only if $p \in \{0,1\}$.
\end{proposition}

\begin{proof}Since $q_1 + q_2 = 0$, admissibility implies $\gcd(q_1,q_2) = 1$, so $q_2 = 1$ and $q_1 = -1$.  If $p=0$, then Wilking \cite{Wi} has already showed his metric is almost positively curved.

Thus, we assume that $p\neq 0$.  We can therefore apply Proposition \ref{prop:fpurpose}, so it is sufficient to determine when $f_{p,-1,1}\leq 0$ on $[0,1]^2$. For these values of $q_1$ and $q_2$, we first compute that $f_{p,-1,1}(\frac{1}{2}, 0) =  4p(2-p)$.  If $p < 0$ or $p > 3$, then $f_{p,-1,1}(\frac{1}{2},0) > 0$. Then Proposition \ref{prop:fpurpose} implies that $\mathcal{E}_{p,-1,1}$ does not have almost positive curvature.

When $p=2$, we have $f_{2,-1,1}(0.3,0.1)  = 0.0069.... > 0$, so Proposition \ref{prop:fpurpose} implies that $\mathcal{E}_{2,-1,1}$ is not almost positively curved.

Finally, when $p=1$, we find $f_{1,-1,1}(x,y) = -4 x^2 y^2 (xy-1)^2 \leq 0$.  Then Proposition \ref{prop:fpurpose} implies that $\mathcal{E}_{1,-1,1}$ is almost positively curved in this case.

\end{proof}

\begin{proposition}\label{prop:twocases}  Suppose that either $q_1 + q_2 < 0$ and $p>q_2$ or that $q_1 + q_2 > 0$ and $p<q_2$.  Then Wilking's metric on $\mathcal{E}_{p,q_1,q_2}$ is not almost positively curved.

\end{proposition}

\begin{proof}We first note that the hypotheses implies that $p-q_1-2q_2 \neq 0$.  Indeed, $p-q_1-2q_2 = (p-q_2)-(q_1 + q_2)$ has the same sign  sign as $-(q_1+q_2)$, being a sum of two such terms.

Now, set $y_0 = -\frac{q_1+q_2}{p-q_1-2q_2}$, which is well-defined since the denominator does not vanish. The argument in the previous paragraph establishes that $y_0 > 0$.

We claim that, in addition, $y_0 < 1$.  Indeed, simple algebra shows that, when $q_1 + q_2 < 0$, that this holds if and only if $p > q_2$, while if $q_1 + q_2 > 0$, that it holds if and only if $p < q_2$.

Now, a simple calculation reveals that $f_{p,q_1,q_2}(0,y_0) = 0$ while for the $x$-partial derivative, that $(f_{p,q_1,q_2})_x(0,y_0) = 4\frac{p_1-q_2}{p-q_1-2q_2} q_1^2 ( q_1+q_2)^2 > 0$.  In particular, it follows that $f_{p,q_1,q_2}(x,y_0)>0$ for $x>0$ sufficiently close to $0$.  Thus, from Proposition \ref{prop:fpurpose}, $\mathcal{E}_{p,q_1,q_2}$ is not almost positively curved.

\end{proof}

\begin{proposition}\label{prop:hessian} Suppose $q_1 + q_2 > 0$ and $p=q_2$.  Then $\mathcal{E}_{p,q_1,q_2}$ does not have almost positive curvature with respect to Wilking's metric.
\end{proposition}

\begin{proof}As admissibility implies $\gcd(q_1,p-q_2) = 1$, we find that $q_1\in \{\pm 1\}$.  If $p = q_2 = 0$, then admissibility implies $(p,q_1,q_2) = (0,1,0)$.  We already showed in Proposition \ref{prop:first3types} that $\mathcal{E}_{0,1,0}$ is not almost positively curved, so we assume for the duration of the proof that $p\neq 0$.  Admissibility then implies $p = q_2> 0$.  We now break into two cases depending on the value of $q_1\in \{\pm 1\}$.

Assume initially that $q_1 = 1$.  Then one easily verifies that $(x,y) = (0,1)$ is simultaneously a zero and a critical point of $f_{p,1,p}$.  That is, $$f_{p,1,p}(0,1) = (f_{p,1,p})_x(0,1) = (f_{p,1,p})_y(0,1) = 0.$$  A routine calculation shows that the eigenvalues of the Hessian matrix at $(0,1)$ are $-2p^2 -2 \pm 2\sqrt{9p^4+24p^3+26p^2+8p+1}$.  We note that \begin{align*}9p^4+24p^3+26p^2+8p+1 & \geq 4 p^4 + 24 + 8p^2\\
&\geq 4 p^4 + 8p^2 + 4\\
&\geq 4 (p^2+1)^2 \\
& \geq 0,
\end{align*} so $-2p^2 -2 + 2\sqrt{9p^4 + 24p^3+26p^2 + 8p + 1}$ is positive.   This eigenvalue has $\vec{v}:=\begin{bmatrix}2(3p^2 + 4p + 1) \\(-2\sqrt{9p^4 + 24p^3 + 26p^2 + 8p + 1} + 4p) \end{bmatrix}$ as an eigenvector.  The top entry of $\vec{v}$ is obviously positive for all $p > 0$.  On the other hand, if $p > 2$, then $\sqrt{9p^4} = 3p^2 > 4p$, so the bottom entry is negative for $p > 2$.  By direct substitution, it is also negative when $p=1$ and when $p=2$.

Thus, $\vec{v}$ points into the fourth quadrant for all $p > 0$.  Since $f_{p,1,p}(0,1) = 0$, it follows that $f_{p,1,p}( t \vec{v}^T) > 0$ for all $t>0$ close enough to zero.  In particular, there is an interior point of $[0,1]^2$ for which $f_{p,1,p}>0$, so $\mathcal{E}_{p,1,p}$ is not almost positively curved by Proposition \ref{prop:fpurpose}.

Having concluded the case where $q_1 = 1$, we now assume that $q_1 = -1$.  Since $q_1 + q_2 > 0$, this implies $p = q_2 \geq 2$.

As in the case where $q_1 = 1$, we observe that $(x,y) = (0,1)$ is simultaneously a zero and a critical point of $f_{p,-1,p}$.  In this case, the eigenvalues of the Hessian matrix at $(0,1)$ are $-2p^2 - 2 \pm 2\sqrt{9p^4 - 24p^3 + 26p^2 - 8p + 1}$.

We now show that $-2 p^2 -2 + 2\sqrt{9p^4 - 24p^3 + 26p^2 - 8p + 1} > 0$ for all $p\geq 2$.  One easily verifies that this inequality holds if and only if $8p^4 - 24p^3 + 24 p^2 -8p>0$. 
%\footnote{one does? how do we get this, {\color{blue} Start with $-2 p^2 -2 + 2\sqrt{9p^4 - 24p^3 + 26p^2 - 8p + 1} > 0$, move the $-2p^2 -2$ over.  Observe both sides are positive.  Now square both sides and rearrange to get $8p^4 - 24p^3 + 24 p^2 -8p>0$.  To go the other way, reverse these steps.  One has to be careful when squarerooting both sides, but our "Observe" sentence allows us to reverse the squaring step via a squareroot.}}
When $p = 2$, this reduces to $16 > 0$, which is obviously true.  On the other hand, for $p \geq 3$, $8p^4 \geq 24 p^3$ and $24p^2 \geq 72 p > 8 p$, so the inequality is also true for $p \geq 3$.

This positive eigenvalue $-2p^2-2+2\sqrt{9p^4 - 24p^3 + 26p^2 - 8p + 1}$ has $v:=\begin{bmatrix}2(3p^2 - 4p + 1)\\ -(2\sqrt{9p^4 - 24p^3 + 26p^2 - 8p + 1} + 4p) \end{bmatrix}$ as an eigenvector.  For $p\geq 2$, the top entry is positive while the bottom is negative, which implies that $v$ points into the fourth quadrant.  Then, just as in the case where $q_1=1$, it follows that $f_{p,-1,p} > 0$ at an interior point of $[0,1]^2$, so
Proposition \ref{prop:fpurpose} implies that $\mathcal{E}_{p,-1,p}$ is not almost positively curved in this case.
\end{proof}

For the next set of parameters, we will make use of the following prior result of the first author and Johnson \cite{DJ} concerning the classical seven-dimensional Eschenburg spaces.  

\begin{theorem}[D-,Johnson]\label{thm:DJ}  Suppose $q_1'$, $q_2'$, and $q_3'$ are pairwise relatively prime integers satisfying all of the following conditions:

\begin{itemize} \item $q_1'+q_2'+q_3' > 0$
\item $q_1' > q_2'$
\item $q_1' > 0$
\item $q_2' q_3'< 0$
\item $q_2' + q_3' \geq 0$
\end{itemize}

Set $\vec{p} = (0,0,q_1'+q_2'+q_3')$ and $\vec{q} = (q_1',q_2',q_3')$.  Then, the particular Wilking metric from \cite{DJ} on $\mathcal{E}_{\vec{p},\vec{q}} = SU(3)\bq S^1_{\vec{p},\vec{q}}$ is not almost positively curved.

\end{theorem}

In this theorem, the notation $SU(3)\bq S^1_{\vec{p},\vec{q}}$ refers to the quotient of $SU(3)$ by the $S^1$-action given by $z\ast B = \diag(1,1,z^{p'}) B \diag(z^{q_1'}, z^{q_2'}, z^{q_3'})^{-1}$.

As we have followed slightly different conventions in this article than those used in \cite{DJ}, we now outline the translation of Theorem \ref{thm:DJ} for our situation.  First, we view the seven-dimensional Eschenburg spaces as biquotients $U(3)\bq S^1_{\vec{p},\vec{q}} \cdot U(1)$, while in \cite{DJ}, they are viewed as biquotient of the form $SU(3)\bq S^1_{\vec{p'},\vec{q'}}$.  Morever, in \cite{DJ}, the metric is obtained via Wilking's doubling trick, where each $SU(3)$ factor is equipped with a Cheeger-deformed metric in the direction of $U(2)$, embedded as $A\mapsto \diag(A,\overline{\det(A)})$.  Supposing $\rho = \begin{bmatrix} 0 & 0 & 1\\ 0 & 1 & 0\\ 1 & 0 & 0\end{bmatrix}$, the map $SU(3)\rightarrow U(3)$ given by $A\mapsto \rho A \rho^{-1}$ induces a diffeomorphism $\Theta$ between $SU(3)\bq S^1_{\vec{p'},\vec{q'}}$ and $U(3)\bq S^1_{p,\vec{q}} U(1)$ when $p = q_1'+q_2'+q_3'$, $q_1 = q_3'$, and $q_2 = q_2'$.  We note that the convention that $q_1' > q_2'$ used in Theorem \ref{thm:DJ} is obtained by applying an isometry of the Eschenburg space.  As such, for our purposes, we actually have $$q_1' = \max\{q_2, p-q_1-q_2\}, q_2' = \min\{q_2, p-q_1-q_2\}, \text{ and } q_3' = q_1.$$

The diffeomorphism $\Theta$ maps points with zero-curvature planes with respect to the metric considered in \cite{DJ} to points with zero-curvature planes with respect to the metric considered by Wilking.   Indeed, if $X',Y'\in \mathfrak{su}(3)$ are such that $\widehat{X'}$, $\widehat{Y'}$ span a horizontal zero-curvature plane in $SU(3)\times SU(3)$, then there are unique real numbers $\lambda_X$ and $\lambda_Y$ for which $\widehat{X}:=\widehat{X' + \lambda_X i I}$ and $\widehat{Y}:=\widehat{Y'+\lambda_Y iI}$ are horizontal in $U(3)\times U(3)$.  Then $\operatorname{span}\{\widehat{X},\widehat{Y}\}$ is a horizontal zero-curvature plane in $U(3)\times U(3)$, which follows easily from the fact that $iI$ lies in the center of $\mathfrak{u}(n+1)$, as well in the center of all of the subalgebras which contribute to Cheeger deformations used in the construction of the metric.

\begin{proposition}\label{prop:DJ} 
Suppose that one of the following occurs:

\begin{itemize}\item $p<q_1 + q_2 < 0$

\item  $0<q_2 < p < q_1 + q_2$

\item  $0<q_1 + q_2 < q_2 < p.$
\end{itemize}  Then Wilking's metric on $\mathcal{E}_{\vec{p},\vec{q}}$ is not almost positively curved except when $p=q_1 + 2q_2$.
\end{proposition}

\begin{proof} We first prove this under the special assumption that the dimension of $\mathcal{E}_{\vec{p},\vec{q}}$ is $7$.  

So, assume initially that $p < q_1 + q_2 < 0$ and recall that $q_2\geq 0$.  If $q_2 = 0$, then admissibility implies $q_1  = 1$, which contradicts the assumption that $q_1 + q_2 < 0$.  Thus $q_2 > 0$ and $q_1 < 0$.  To use Theorem \ref{thm:DJ}, we must have $p' > 0$, so we use the equivalent action obtained by negating all of $p$, $q_1$, $q_2$.

We thus have $q_2 < 0 < q_1 + q_2 < p$ and $q_1 > 0$.  Then $p' = p > 0$, $q_3' = q_1 > 0$.  Moreover, we have $p-q_1-q_2 > 0 > q_2$, so that $q_1' = p-q_1-q_2$ and $q_2' = q_2$.

Then, we have $p' > 0$, $q_1' > 0 > q_2'$, $q_2'q_3' = q_2 q_1 < 0$, and $q_2' + q_3' = q_2 + q_1 > 0$, so case 1 of Theorem \ref{thm:DJ} implies $\mathcal{E}^7_{\vec{p},\vec{q}}$ is not almost positively curved in this case.

Next, assume that $0<q_2 < p < q_1 + q_2$.  Then $p-q_1-q_2 < 0 < q_2$, so $q_1' = q_2$ and $q_2' = p-q_1-q_2$.  Since $p' = p >0$ and $q_3' = q_1 >0$, we find that case 1 of Theorem \ref{thm:DJ} implies $\mathcal{E}^7_{\vec{p},\vec{q}}$ is not almost positively curved.

Finally, assume that $0<q_1 + q_2 < q_2 < p$.  Then $q_3' = q_1 < 0$.  We have $\{q_1',q_2'\} = \{q_2, p-q_1-q_2\}$.  We briefly note that it is not possible for $q_2 = p-q_1-q_2$, for then admissibility implies $q_2 = 1$, contradicting the fact that both $q_1 < 0$ and $q_1 + q_2 > 0$.

As both $q_2, p-q_1-q_2 > 0$, we see that $q_2'q_3' < 0$ in either case.  We now break into cases depending on the value of $q_2'\in \{q_2, p-q_1-q_2\}$.

Suppose initially that $q_1'=q_2  > p-q_1-q_2 = q_2'$.  Then $q_2' + q_3' = p-q_1-q_2 + q_1 = p-q_2 > 0$, so $\mathcal{E}^7_{\vec{p},\vec{q}}$ is not almost positively curved.

Finally, suppose that $q_1'=p-q_1-q_2 > q_2 = q_2'$.  Then $q_2'+q_3' = q_2 + q_1 >0$, so $\mathcal{E}^7_{\vec{p},\vec{q}}$ is not almost positively curved. 

\bigskip

Having completing the proof when $\dim \mathcal{E}_{\vec{p},\vec{q}} = 7$, we now turn to the general case.  In fact, we will show that for fixed $\vec{p},\vec{q}$, the existence of an open set of of points with zero-curvature planes on $\mathcal{E}^7_{\vec{p},\vec{q}}$ implies the same on $\mathcal{E}^{4n-1}_{\vec{p},\vec{q}}$ for all $n\geq 2$.

Since we have already established that $\mathcal{E}^7_{\vec{p},\vec{q}}$ is not almost positively curved, Corollary \ref{cor:checkF} yields a non-empty open set of points $U\subseteq \mathcal{F}$ with the following property:  for each $(g_1,g_2)\in U$, there is a horizontal zero-curvature plane $\operatorname{span}\{\widehat{X},\widehat{Y}\}$.  In particular, each condition of Proposition \ref{prop:Wcurvature} is verified by $X$ and $Y$.

As $\dim \mathcal{E}^7_{\vec{p},\vec{q}} = 7$, $X$ and $Y$ are $3\times 3$ matrices.  Now, for each $n\geq 3$, let $X'$ be the unique $(n+1)\times (n+1)$ matrix whose top left $3\times 3$ block is $X$, and all other entries are zero.  Suppose $Y'$ is obtained similarly from $Y$.  Then it is straightforward to verify that $X'$ and $Y'$ satisfy the conditions of Proposition \ref{prop:Wcurvature}, so that $\widehat{X'}$ and $\widehat{Y'}$ span a horizontal zero-curvature on an open set of $\mathcal{F}$.  Thus, by Corollary \ref{cor:checkF}, $\mathcal{E}^{4n-1}_{\vec{p},\vec{q}}$ is not almost positively curved in this case.

\bigskip

\end{proof}

\begin{proposition}\label{prop:Ralg}
Suppose that $p\geq q_1 + q_2 > 0$, $p> q_2$, and $q_1\geq 0$.  Assume in addition that $(p,q_1,q_2) \neq (1,1,0)$.  Then the Wilking metric on $\mathcal{E}_{\vec{p},\vec{q}}$ is almost positively curved.
\end{proposition}

\begin{proof}
Because $(p,q_1,q_2) \neq (1,1,0)$, Proposition \ref{prop:fpurpose} implies that we need only demonstrate that $f=f_{p,q_1,q_2} \leq 0$ on the rectangle $[0,1]^2$.

We will first deal with a few special cases before moving onto the general case.

We begin with the case where $q_1 = 0$.  Then admissibility implies $q_2 = 1$ and $\gcd(0,p-1) = 1$, so $p\in \{0,2\}$.  Since $p > q_2$, we must have $p=2$.  Then $f_{2,0,1} = -4x^2y^2 (xy-1)^2 \leq 0$.

We now turn to the case where $p = q_1 + 2q_2$.  In this case, $f$ simplifies to $$-(xy - 1)^2((xy - 1)q_1^2 + q_2(xy - 1)q_1 - 2xq_2^2y)^2,$$ which is obviously non-positive.

We henceforth assume that $p\neq q_1 + 2q_2$ and that $q_1 > 0$.

To show that $f\leq 0$, we will show that $f\leq 0$ on the boundary of the rectangle and that $f$ has no critical points in the interior of the rectangle.

We begin with the boundary.  We compute \begin{align*} f(0,y) &= -q_1^2\left((1 - y)q_1 + (-2q_2 + p)y + q_2\right)^2 \\
f(1,y) &= -\left(p(-2q_1 - q_2 + p)y^2 + (-pq_2 + 2q_1^2 + 3q_1q_2)y - q_1(q_1 + q_2)\right)^2 \\
f(x,0) &= -q_1^2(q_1 + q_2)^2  \\
f(x,1) &= -\left(p(q_1 - q_2)x^2 + (q_1^2 + (-4p + 3q_2)q_1 + p(-q_2 + p))x + q_1(-q_2 + p)\right)^2.\end{align*}  In particular, along each boundary segment of the rectangle, $f$ is obviously non-positive.

We now show there are no critical points in the interior.  We, again, will handle a few special cases first, beginning with the case where $q_2 = 0$.  Admissibility and the fact that $p\geq q_1 + q_2$ imply $(p,q_1,q_2) = (2,1,0)$.  Here we can argue directly that $f$ has no critical points in the interior of the rectangle.  Indeed, the partial derivative $f_x$ factors as\begin{equation}
    -2 y \left(2  x^{2} y^{2} - 2 x y^{2} - x y - y + 3\right) \left(4 x y - 2 y - 1\right).
\end{equation}

Since we are interested in interior critical points, we find that either $2  x^{2} y^{2} - 2 x y^{2} - x y - y + 3 = 0$ or $4 x y - 2 y - 1 = 0$.  For the first possibility, we find 
%\footnote{i matched the -2 and 3 up with the second and third terms in the expression above them, but i don't know how the first term $2y^2x(x-1)$ gets us to $-\frac{1}{2}$ {\color{blue}  The parabola $x(x-1)$ is nonpositive on $[0,1]$ and has a vertex at $(1/2,-1/4)$, so its min is $-1/4$.  The max of $y$ is $1$, so the min of $2y^2x(x-1)$ is $2\cdot 1 \cdot -1/4 = -1/2$. }}
\begin{align*}2  x^{2} y^{2} - 2 x y^{2} - x y - y + 3 &=  2y^2x(x-1) -y(x+1) + 3\\ &\geq -\frac{1}{2} - 2 + 3 \\ & > 0,\end{align*} 
giving a contradiction.  On the other hand, if $4xy-2y-1 = 0$, then $x\neq \frac{1}{2}$ and $y = \frac{1}{4x-2}$.  Substituting this into $f_y = 0$, we find $$\frac{5x^3 - 5x^2 - x + 1}{(2x - 1)^3} = 0.$$  As $x\neq \frac{1}{2}$, we find $x = 1$ or $x= \pm \frac{1}{\sqrt{5}}$.  In order to have an interior critical point, we must have $x = \frac{1}{\sqrt{5}}$, which then implies that $y = \frac{1}{4(1/\sqrt{5})-2} \approx -4.73...$, which is not interior.

Next, we handle the special case where $p=q_1+q_2$. Admissibility and the hypothesis that $q_1 > 0$ imply $q_2= q_1 = 1$.  We will directly show that $f_{2,1,1}(x,y)$ has no critical points in the interior of the rectangle $[0,1]^2$.

To see this, we note that $f_x = -4y(2xy^3 - 8xy^2 + 8xy + y^2 - 2)$.  Since we are interested in interior points, we find that $x =-\frac{y^2-2}{2y(y-2)^2}$.  Substituting this into $f_y$ and solving for $y$, we find that $y \in \{1, \frac{5\pm \sqrt{5}}{2}\} \approx \{ 1, 1.38...,3.61...\}$.  Thus, $y$ is not in the interior of the rectangle, so $f$ has no interior critical points in this case as well.

For the remainder of the proof, we assume $q_2 > 0$ and $p\neq q_1 + q_2$.  Since $p\geq q_1+q_2$ by assumption, we will, in fact, assume that $p > q_1 + q_2$.

We suppose for a contradiction that $f_x(x_0,y_0) = f_y (x_0,y_0) = 0$ for some $(x_0,y_0) \in (0,1)^2$.  Then $(x_0,y_0)$ is a zero of the function $xf_x(x,y) - y f_y(x,y)$.  This function has the expression $$-2(p-q_1-2q_2) y g_{p,q_1,q_2}(x,y)$$ where $g(x,y) =g_{p,q_1,q_2}(x,y) = $  \begin{multline}p^2(q_1 - q_2 )x^3y^3 + (- p^3  + p^2 q_1  + 2 p^2 q_2)x^2 y^3  + p(3 p q_1  - p q_2 - 6 q_1^2 - 2  q_1 q_2) x^2y^2\\ 2p(- p q_1  +  q_1^2  + 2 q_1q_2) xy^2 + q_1(2 pq_1  - 2pq_2 + q_1^2 - q_1 q_2) xy \\+ q_1^2(- p  + q_1 + 2 q_2)y - q_1^3 - q_1^2q_2.\end{multline}

In particular, since $(x_0,y_0)$ is a zero of $xf_x - y f_y$ and $p\neq q_1+2q_2$, it follows that $g(x_0,y_0) = 0$.  We now show that this is impossible by showing that  $g< 0$ on the boundary of $[0,1]^2$ and that $g$ has no critical points in the interior of the rectangle $[0,1]^2$.

We begin by showing $g$ has no critical points in the interior.  We compute $$xg_x(x,y) - y g_y(x,y) = y(p xy + q_1)^2(p - q_1 - 2q_2).$$  Since $p\neq q_1 + 2q_2$ by assumption, this can only vanish in the interior of the rectangle if $pxy = -q_1$.  But since $q_1> 0$ and $p> 0$, there is no solution.  In particular, $xg_x - y g_y$ is non-vanishing in the interior of the rectangle.  As any interior critical point of $g$ is a zero of $x g_x - y g_y$, we conclude that $g$ has no critical points in the interior.

It remains to verify that $g < 0$ on the boundary of the rectangle $[0,1]^2$.  We begin by verifying that $g<0$ on the four corners of the rectangle.  We compute

\begin{itemize}\item $g(0,0) = -q_1^2(q_1+q_2)$
\item $g(0,1) = -q_1^2(p-q_2)$
\item  $g(1,0) = -q_1^2(q_1+q_2)$
\item $g(1,1) = -(p-q_1)^3$.

\end{itemize}

In all four cases, the fact that $p > q_1 + q_2$ and that both $q_1,q_2 > 0$ implies $g<0$ at all four corner points.

We now turn to verifying $g<0$ on the interior of each of the four line segments making up the boundary of the rectangle.  First, we compute $g(x,0) = -q_2^2(q_1+q_2) < 0$.

Next, we focus on $g(0,y)$ with $y \in (0,1)$.  We find $$g(0,y) = -q_1^2((1 - y)q_1 + (-2q_2 + p)y + q_2).$$  Since the factor in parenthesis is linear as a function of $y$, to verify that $g(0,y) < 0$ for all $y\in [0,1]$, we only need to verify it when $y \in \{0,1\}$.  But these correspond to corner points, so the proof is complete in this case.

We next consider $g(1,y)$, which is given by $g(1,y) =$ \begin{multline} -p^2 (-2q_1 - q_2 + p_1)y^3 + \left(-4q_1^2 + (2q_2 + p)q_1 - q_2p\right)p y^2 \\ + q_1 \left(2q_1^2 + (q_2 + p)q_1 - 2q_2p\right)y - q_1^2(q_1 + q_2).  \end{multline}

We claim that for fixed $y\in (0,1)$, $g_{p,q_1,q_2}(1,y) < g_{p,q_1,0}(1,y)$.  To prove this, it is sufficient to show that $g_{p,q_1,q_2+1}(1,y) < g_{p,q_1,q_2}(1,y)$.  We compute that $$g_{p,q_1,q_2+1}(1,y)-g_{p,q_1,q_2}(1,y) = (y-1)(yp+q_1)^2.$$  Since $y-1<0$ and $yp+q_1 \neq 0$, we see that this expression is always negative.

So, to prove that $g_{p,q_1,q_2}(1,y)<0$, we now only need to show that $g_{p,q_1,0}(1,y)\leq 0$.  But $g_{p,q_1,0}(1,y) = -(p y - q_1)^2((-2q_1 + p)y + q_1)$.  Observe that the linear factor $(-2q_1 + p)y+q_1$ ranges between $q_1 > 0$ and $p-q_1 \geq 0$, so it is non-negative for $y\in (0,1)$.    Since $(py-q_1)^2\geq 0$, it follows that $g(1,y)\leq 0$ for all $y\in [0,1]$ when $q_2 = 0$, completing the proof in this case.

We finally consider $g(x,1)$, where \begin{multline} g(x,1) = p^2(q_1 - q_2)x^3 - (p^2 + (-4 q_1 - q_2)p + 6q_1^2 + 2q_1 q_2)px^2\\ - q_1(2p^2 + 2(-2q_1 - q_2)p - q_1(q_1 - q_2))x - q_1^2(-q_2 + p).\end{multline}

Observe that the condition $p> q_1 + q_2$ implies $q_2 < p-q_1$.   We claim  $g_{p,q_1,q_2}(x,1)< g_{p,q_1, p-q_1}(x,1)$ for fixed $x\in (0,1)$.  To establish this, it is sufficient to show that $g_{p,q_1,q_2}(x,1) < g_{p,q_1,q_2+1}(x,1)$, but $$g_{p,q_1,q_2}(x,1)-g_{p,q_1,q_2+1}(x,1) = (x-1)(px+q_1)^2.$$  Since $x\in (0,1)$ and $p,q_1 >0$, it follows that $(x-1)(px+q_1)^2 < 0$, establishing the claim.

So, to prove that $g_{p,q_1,q_2}(x,1) < 0$, it is enough to show that $g_{p,q_1,p-q_1}(x,1) \leq 0$.  But a straightforward calculation shows that $$g_{p,q_1,p-q_1}(x,1) = -(px - q_1)^2((-2q_1 + p)x + q_1).$$  The linear factor varies between $q_1 >0$  and $p-q_1 > 0$, so $g_{p,q_1,p-q_1}(x,1) \leq 0$.  This completes the proof that $g_{p,q_1,q_2}(x,1)<0$ for any $q_2 < p-q_1$, which, in turn, completes the proof of the proposition.

\end{proof}

\section{Free circle quotients of generalized Eschenburg spaces}\label{sec:t2}

A general $U(n-1)\times T^2$ biquotient action on $U(n+1)$ which normalizes the $U(n)\times T^3$-action of Proposition \ref{prop:cohom2} has the form $(C,z,w)\ast B = \diag(z^p w^t,1,...,1) B \diag(z^{q_1} w^{s_1}, z^{q_2} w^{s_2}, C)^{-1}$ where all the exponents are arbitrary integers.   It is well known that the unique such example with $z^p w^t  = 1$ for all $z$ and $w$ is the homogeneous space $U(n+1)/T^2\times U(n-1)$ and is diffeomorphic to the projectivized unit tangent bundle of $\mathbb{C}P^n$.  As Wilking showed this space to be almost positively curved \cite{Wi}, we will henceforth assume $z^p w^t\neq 1.$  By reparamaterizing the $T^2$, we may assume  $t=0$, so $p\neq 0$.  We assume without loss of generality that the action is effective, which obviously implies $\gcd(p,q_1,q_2) = \gcd(s_1,s_2) = 1$.

%\begin{proposition}  Suppose $p\neq 0$ and that the action is effective.  Then $\gcd(p,q_1,q_2) = \gcd(s_1,s_2) = \gcd(p,q_1 s_2 - q_2 s_1) = 1$.

%\end{proposition}

%\begin{proof}  By looking at the sub-actions where either $z=1$ or $w=1$, one finds $\gcd(p,q_1,q_2) = \gcd(s_1, s_2) = 1$.

%We first assume that $q_1 s_2 - q_2 s_1 = 0$.  In this case, taking $(z,w) = (u^{s_2},u^{-q_2})$ or $(z,w) = (u^{s_1}, u^{-q_1})$ with $u$ a $p$-th root of unity gives an element of the kernel of the action.  Since the action is effective, we conclude that $p$ divides each of $s_1,q_2,s_1,q_1$.  In particular, $p|\gcd(p,q_1,q_2) = 1$, so that $p = \pm 1$.  Then $\gcd(p,q_1s_2 - q_2 s_1) = 1$, as claimed.

%We henceforth assume that $q_1 s_2 - q_2 s_1 \neq 0$.  Set $d = \gcd(p,q_1 s_2 - q_2s_1)$ and note that since $p\neq 0$, $d\neq 0$.  Assume for a contradiction that $d\geq 2$.  

%Since $\gcd(s_1,s_2) = 1$, there are integers $k_1,k_2$ for which $k_1 s_2 - k_2 s_1 = \frac{q_1s_2-q_2 s_1}{d}$.  Set $$z = e^{ \frac{2\pi i}{q_1 s_2 - q_2 s_1}(k_1 s_2 - k_2 s_1)} \text{ and } w = e^{\frac{2\pi i}{q_1s_2 - q_2 s_1}(k_2 q_1 - k_1 q_2)}.$$
%Since $d\geq 2$, $z\neq 1$.  On the other hand, it is easy to verify that $z^p = z^{q_1} w^{s_1}= z^{q_2} w^{s_2} = 1$, so that the action is not effective.

%\end{proof}

\begin{lemma}\label{lem:freeT2action}  An action $$(C,z,w) \ast B = \diag(z^p,1...,1) B \diag(z^{q_1} w^{s_1}, z^{q_2}, w^{s_2}, C)^{-1}$$ is free if and only if $$|q_1 s_2 - q_2s_1| = |(q_1-p)s_2-q_2s_1| = |q_1s_2-(q_2-p)s_1| = 1.$$
\end{lemma}

\begin{proof}

We begin with the backwards implication.  So, assume that $$|q_1 s_2 - q_2s_1| = |(q_1-p)s_2-q_2s_1| = |q_1s_2-(q_2-p)s_1| = 1$$ and that $(C,z,w)\in U(n-1)\times T^2$ fixes $B\in U(n+1)$. Then, we find that $\diag(z^p,1...,1) = B \diag(z^{q_1} w^{s_1}, z^{q_2} w^{s_2}, C) B^{-1}$, so $\diag(z^p,1..,1)$ and $\diag(z^{q_1} w^{s_1}, z^{q_2} w^{s_2}, C)$ are conjugate.  Thus, these two matrices have the same multiset of eigenvalues.  So, either $z^{q_1} w^{s_1} = z^{q_2} w^{s_2} = 1$ or $$\{z^{q_1} w^{s_1}, z^{q_2} w^{s_2}\} = \{1, z^p\}.$$  In the first case, we find that $(z^{q_1} w^{s_1})^{s_2} = (z^{q_2} w^{s_2})^{s_1}$.  Dividing, we then conclude $z^{q_1 s_2 - q_2s_1} = 1$.  As $|q_1 s_2 - q_2s_1| = 1$, we find $z=1$.  Then $w^{s_1} = w^{s_2} = 1$ so $w^{\gcd(s_1,s_2)} = 1$.  But, since $|q_1s_2 - q_2s_1| = 1$, we have $\gcd(s_1,s_2) = 1$, so $w= 1$ as well.  The case where $\{z^{q_1} w^{s_1}, z^{q_2}w^{s_2}\} = \{z^p, 1\}$ works similarly using $|(q_1-p)s_2-q_2s_1|  = |q_1 s_2 - (q_2-p)s_1| = 1$.

We now prove the forwards implication.  Assume that the action is free, (so it is also effective) and, for a contradiction, assume at least one of $|q_1 s_2 - q_2s_1|,  |(q_1-p)s_2-q_2s_1| , |q_1s_2-(q_2-p)s_1| $ is not $1$.  We will assume, in particular, that $|q_1 s_2 - (q_2 - p)s_1| \neq 1$, the other two possibilities being very similar.  If $|q_1 s_2 - (q_2 - p)s_1|= 0$, then setting $(z,w) = (u^{s_2}, u^{p-q_2})$ for any $u\in S^1$, we find that $(I,z,w)\ast B_0 = B_0$ where $B_0$ is the matrix obtained by swapping the first two rows of the identity matrix $I$.  Thus, freeness implies $(u^{s_2}, u^{p-q_2}) = (1,1)$ for all $u\in S^1$, so $s_2 = p-q_2 = 0$.  As $\gcd(s_1,s_2) = 1$ and $s_2= 0$, $|s_1| = 1\neq 0$.  We now note that for $(z,w) = (u^{s_1}, u^{-q_1})$, $(A,z,w)\ast B_0 = B_0$.  Since $s_1\neq 0$, we have contradicted the fact that the action is free.

We next assume that $d:=|q_1 s_2 - (q_2 - p)s_1| \geq 2$.  Then setting $(z,w) = (e^{\frac{2\pi i}{d} s_2}, e^{\frac{2\pi i}{d}(p-q_2)}$, we find that $B_0$ is fixed.  Freeness implies that $d|s_2$ and $d|(q_2-p)$.  Then setting $(z,w) = (e^{\frac{2\pi i}{d} s_1}, e^{-\frac{2\pi i}{d} q_1})$, we find that $B_0$ is fixed.  Because the action is free, $d|s_1$ and $d|q_1$.  As $d|s_1$ and $d|s_2$, we have contradicted the fact that $\gcd(s_1,s_2)=1$.
\end{proof}

\begin{proposition}\label{prop:t2}
    Assume $ p \not= 0$.  Then up to equivalence of action, there are at most two free $T^2$-actions on $U(n+1)/U(n-1)$, given by $$(p,q_1,q_2,s_1,s_2)\in \{ (2,0,1,1,0), (2,0,1,-1,1)\}.$$
\end{proposition}

\begin{proof}

    Assume that $p\neq 0$ and that the action is free.  From Lemma \ref{lem:freeT2action},

    \begin{equation}\label{eqn:abs1}
        |q_1 s_2 - q_2s_1| = 1
    \end{equation}
    \begin{equation}\label{eqn:abs2}
        |(q_1-p)s_2-q_2s_1| = 1
    \end{equation}
    \begin{equation}\label{eqn:abs3}
        |q_1s_2-(q_2-p)s_1| = 1
    \end{equation}
    Subtracting \ref{eqn:abs1} and \ref{eqn:abs2}, we learn 
    \begin{equation}\label{eqn:abs4}
        p s_2 \in \{0, \pm 2\},
    \end{equation}
    and doing the same with \ref{eqn:abs1} and \ref{eqn:abs3} yields
    \begin{equation}\label{eqn:abs5}
        p s_1 \in \{0, \pm 2\}.
    \end{equation}
    Since $p \not= 0$, the only way for either expression to be 0 is if $s_1 = 0$ or $s_2 = 0$. From the assumption that $\gcd(s_1, s_2) = 1$, only one of $s_1$ and $s_2$ can be zero, and the other must be $\pm 1$. 

    If $s_1 = 0$, by reparameterization of the action we may assume $s_2 = 1$.  Moreover, replacing the coordinates $(z,w)$ with $(z,w')$, where $w' = z^{q_2}w$, we may assume $q_2 = 0$.  With $q_2 = 0$, we may further assume $q_1 \geq 0$.
    
    We now find from equation \eqref{eqn:abs2} that $q_1-p = \pm 1$ and from \eqref{eqn:abs1} and \eqref{eqn:abs3} that $q_1 =  1$. Since $p \not= 0$, we find $p = 2q_1 = 2$, yielding $(p,q_1,q_2,s_1,s_2) = (2,1,0,0,1)$. Analogously, if $s_1 =  1$ and $s_2 = 0$, we find $(p,q_1,q_2,s_1,s_2) = (2,0,1,1,0)$.   We note, however, that these two actions are equivalent.

    We may now assume that neither $s_1$ nor $s_2$ is zero.  Since $\gcd(s_1,s_2) = 1$, equations \eqref{eqn:abs4} and \eqref{eqn:abs5} imply that $p = \pm 2$ and $|s_1|=|s_2| = 1$.  We may again assume $s_2 = 1$.  We now subtract \ref{eqn:abs2} and \ref{eqn:abs3} to see that $p(s_1 + s_2) \in \{0, \pm2\}$ and conclude that $s_1 = -s_2 = -1$.   Equations \eqref{eqn:abs1} and \eqref{eqn:abs2} now imply that $|q_1 + q_2| = 1$ and that $p$ and $q_1 + q_2$ have the same sign.  Using the new coordinates $(z,w')$ where $w' =wz^{-q_1}$, we find $(p,q_1,q_2,s_1,s_2) = ( 2,0,1,-1,1)$ or $(-2,0,-1,-1,1)$.  The two actions are equivalent.

    Conversely, it is easy to see that each of these choices for $(p, q_1, q_2, s_1, s_2)$ satisfies \ref{eqn:abs1}, \ref{eqn:abs2}, and \ref{eqn:abs3}, so the action is free by Lemma \ref{lem:freeT2action}.
\end{proof}

\begin{theorem}Suppose $T^2$ acts on $U(n+1)/U(n-1)$ by either of the two actions in Proposition \ref{prop:t2}.  Then the quotient space admits a metric of almost positive sectional curvature.

\end{theorem}

\begin{proof}Consider the generalized Eschenburg space $\mathcal{E}_{2,0,1}$.  This space admits a Wilking metric of almost positive sectional curvature by Theorem \ref{thm:main}.

For both $T^2$-actions given by $$(p,q_1,q_2,s_1,s_2) \in\{ (2,0,1,1,0), (2,0,1,-1,1)\},$$ we consider the sub-action by the $S^1\subseteq T^2$ defined by $w=1$.  This $S^1$ acts on $U(n+1)/U(n-1)$ with quotient $\mathcal{E}_{2,0,1}$.  It follows that $M$ is a free isometric $S^1$-quotient of $\mathcal{E}_{2,0,1}$.  Hence, by O'Neill's formula \cite{On1}, the induced metric is almost positively curved.

%Now, assume for a contradiction that $M$ is homotopy equivalent to a homogeneous space $N$.  Observe that from the long exact sequence of homotopy groups, that $M$ is simply connected with $\pi_2(M)\cong \mathbb{Z}^2$.  Thus, by the Hurewicz theorem and the universal coefficients theorem, $H^2(M)\cong \mathbb{Z}^2$.  The same must be true of $N$.  

%Then $\mathcal{E}$ is a principal $S^1$-bundle over $M$, and hence, is classified by its Euler class $e\in H^2(M)$.  The element $e$ is primitive since $\mathcal{E}$ is simply connected.

%Via the homotopy equivalence $M\sim N$, we obtain an element $e'\in H^2(N)$ which must also be primitive.  Hence, it corresponds to a principal $S^1$-bundle $N'$ over $N$ with $N'$ simply connected.    As $N$ is simply connected and homogeneous, it is a homogeneous space of a semi-simple Lie group.  From \cite{PS}, it follows that $N'$ is homogeneous as well.

%Now, we apply Lemma \ref{lem:s1homotopy} to find that $\mathcal{E}$ and $N'$ are homotopy equivalent.

%On the other hand, we now prove that $\mathcal{E}$ is strongly inhomogeneous.  From Proposition \ref{prop:}, we know that if $\mathcal{E}$ is homotopy equivalent to a homogeneous space, then it must be homotopy equivalent to a homogeneous generalized Eschenburg space.  From Proposition \ref{prop:top}, $H^{2n}(\mathcal{E})\cong \mathbb{Z}/2\mathb{Z}$.

%For a homogeneous generalized Eschenburg space $\mathcal{E}'$, the order of $H^{2n}(\mathcal{E}')$ is $\left| \frac{q_1^{n+1} - q_2^{n+1}}{q_1-q_2}\right| = | \sum_{i=0}^{n+1} q_1^i q_2^{n+1-i}$.
\end{proof}

\section{Quasi-positive curvature in generalized Eschenburg spaces}\label{sec:quasi}

We now prove one of the main results of \cite{longkerinpaper} with the corrected type (iv) planes.

\begin{theorem}[Kerin]\label{thm:quasi}  Every generalized Eschenburg space admits a metric of quasi-positive curvature.

\end{theorem}

In \cite{longkerinpaper}, type (iv) zero-curvature planes are only used in two places: in \cite[Lemma 5.3]{longkerinpaper} and in the verification of quasi-positive curvature in the special case of $\mathcal{E}_{\vec{p}_0,\vec{q}_0}$ where $\vec{p}_0 = (1,-1,-1,...,\pm 1)\in \{-1,1\}^{n+1}$, the number of times $-1$ appears is either $n$ or $\lfloor \frac{n+2}{2}\rfloor$, and $\vec{q}_0 = (0,0)$.  We provide new proofs for these cases, so that, in particular, \cite[Theorem B]{longkerinpaper} is still correct.

     We let $A(r) := A(\pi/2,r)\in \mathcal{F}$.  Kerin shows that except when $q_1 = q_2 = 0$ and $\vec{p}$ is a permutation of $\vec{p}_0$, the entries of $\vec{p}$ and $\vec{q}$ can be permuted so that $\mathcal{E}_{\vec{p},\vec{q}}$ has positive curvature at any $[(A(r),I)]$ for which $\cos^2 r$ is irrational.  However, as mentioned above, the type (iv) planes he works with are not general enough.  The following proposition, whose proof was provided by Kerin, therefore establishes that all Eschenburg spaces except $\mathcal{E}_{(1,-1,...,-1), (0,0)}$ admit quasi-positively curved metrics.

%%%%%%%%%%%%%%%%%%%%%%%%%%%%%%%%%%%%%
\begin{proposition}[Lemma 5.3 \cite{longkerinpaper}]\label{prop:fix}  Suppose $p_1\neq p_2$ and $p_1 + p_2 \neq q_1 + q_2$.  If $r$ is chosen so that $\cos r$ is transcendental then there are no zero-curvature planes of type (iv).
\end{proposition}

\begin{proof}  Suppose there is such a zero-curvature plane.   Then Proposition \ref{prop:pregiveseqns} applies.  Kerin's proof correctly handles the cases where $V=0$, $W=0$, or when $y_j=0$ for all $j\geq 3$.  After rescaling $X$ if necessary, we may assume $V = W$ and that some $y_j\neq 0$.  Without loss of generality, we assume $y_3\neq 0$.  Then Proposition \ref{prop:pregiveseqns} gives the following equations:

    \begin{equation}\label{eqn:MK4.1}
        -{\left(\beta - 1\right)} \left( p_{1} - p_{2}\right) \cos^{2}r + \beta p_{1} - \beta q_{2} + p_{2} - q_{1} = 0
    \end{equation}
    \begin{equation}\label{eqn:MK4.2}
        \lambda_1 \alpha p_{1} \cos^{2}r + \lambda_1 \alpha p_{2} \sin^{2}r + 2 {\left(p_{1} - p_{2}\right)} \cos r \im x_{2}  \sin r - \lambda_1 \alpha q_{1} = 0
    \end{equation}
    \begin{equation}\label{eqn:MK4.3}
        -i \lambda_1 \alpha \cos r \sin r + x_{2} \cos^{2} r + \overline{x_{2}} \sin^{2} r = i{\left(\beta - 1\right)} \cos r \sin r
    \end{equation}
    \begin{equation}\label{eqn:MK4.4}
        x_{3} \cos r = y_{3} \sin r
    \end{equation}

   Then equations \eqref{eqn:4.9} and \eqref{eqn:MK4.4} together give $x_2 = i\tan r.$   Substituting this value into equation \ref{eqn:MK4.2}, we solve for $\lambda_1\alpha$, finding $$\lambda_1 \alpha = \frac{-2(p_1-p_2)\sin^2 r}{p_1\cos^2 r+p_2\sin^2 r - q_1}.$$  We note that since $p_1\neq p_2$ by assumption, the denominator of $\lambda_1 \alpha$ is not identically zero.
    
    Now, we substitute the formulas for $x_2$ and $\lambda_1\alpha$ into equation \eqref{eqn:MK4.3} and solve for $\beta$; we then substitute this formula for $\beta$ into \eqref{eqn:MK4.1} to obtain a rational equation of the form $\frac{g}{h} = 0$ where
    \begin{itemize}
        \item $g = a\cos^4(r) + b\cos^2 r + c$ with 
        \begin{itemize}
            \item[•] $a = (p_{1} - p_{2}) (q_{1} - q_{2})$,
            \item[•] $b = p_{1}^{2} + 3 p_{1} p_{2} - {\left(4 p_{1} + p_{2}\right)} q_{1} + q_{1}^{2} - {\left(p_{1} + 2 p_{2} - 3 q_{1}\right)} q_{2}$, and
            \item[•] $c = (q_1-p_2)(p_1-q_2)$; and
        \end{itemize}
        \item $h = {\left(p_{1} - p_{2}\right)} \cos^{4}\left(r\right) + {\left(p_{2} - q_{1}\right)} \cos^{2}\left(r\right)$.
    \end{itemize}

    By assumption, $h$ is not identically zero.  So, the vanishing of $\frac{g}{h}$ implies the vanishing of $g$.  Since $ax^4 + bx^2 + c$ is a polynomial  with integer coefficients and the transcendental number $\cos r$ is a zero of it, we conclude that $a=b=c= 0$.  Since $p_1\neq p_2$, the fact that $a=0$ implies that $q_1 = q_2$.  The fact that $c=0$ implies either $p_2 = q_1$ or $p_1 = q_2$.  By permuting the $p_i$, we may assume that $p_2 = q_1$.  Further, by subtracting $q_1$ from all entries of both $\vec{p}$ and $\vec{q}$, we may assume that $p_2 = q_1 = 0$.  Then $0 = b = p_1(p_1-q_2)$.  Admissibility implies $1=\gcd(p_2-q_1, p_1 - q_2)   = \gcd(0,p_1-q_2)$, so $|p_1-q_2| = 1$.  Then $0 = p_1(p_1-q_2)$ implies $p_1 = 0$, which contracts the fact that $p_1 = p_2$.

\end{proof}

We conclude this section by considering the Eschenburg space $\mathcal{E}_{\vec{p}_0, \vec{q}_0}$.  At the end of \cite[Section 5]{longkerinpaper}  Kerin verifies that for the matrix $A_0:=A(3\pi/4,\pi/4)$, $[(A_0,I)]$ is a point of positive curvature for this Eschenburg space.  We verify it for the more general type (iv) plane now, completing the proof of Theorem \ref{thm:quasi}.

\begin{proposition}The Eschenburg space $\mathcal{E}_{\vec{p}_0, \vec{q}_0}$ has positive curvature at the point $A_0$.
\end{proposition}

\begin{proof}

For this choice of $(r,t)$ and $\vec{p}_0,\vec{q}_0$, the first two equations of Proposition \ref{prop:pregiveseqns} yield \begin{equation}\label{eqn:last1} \sqrt{2} \im (y_{3}) - \frac{1}{2}   \beta - \frac{1}{2} = 0\end{equation} and \begin{equation}\label{eqn:last2}-\frac{1}{2}   \lambda_1 \alpha - \sqrt{2} \im (x_{3}) - \im (x_{2}) = 0.\end{equation}

Following the same approach as in proofs of Propositions \ref{prop:caseivsubcase1.1}, \ref{prop:caseivsubcase1.2}, and \ref{prop:caseivsubcase2}, we conclude that $\mathcal{E}_{\vec{p}_0,\vec{q}_0}$ has no zero-curvature planes of type (iv) except possibly when $V=W$ and $y_j = 0$ for all $j\geq 4$.

Then the third condition of Proposition \ref{prop:pregiveseqns} gives the following equations.

\begin{equation} \label{eqn:last3}-\frac{1}{4} i  \sqrt{2} \beta + \frac{1}{4} i  \sqrt{2} + \frac{1}{2}  \overline{y_{3}} = \frac{1}{4} i  \sqrt{2} \lambda_1 \alpha + \frac{1}{2}  \sqrt{2} \re(x_{2}) + \frac{1}{2}  \overline{x}_{3}\end{equation}

\begin{equation} \label{eqn:last4} -\frac{1}{2}  \sqrt{2} \re(y_{3}) + \frac{1}{4} i  \beta + \frac{1}{4} i = \frac{1}{4} i  \lambda_1 \alpha + \frac{1}{2}  \sqrt{2} \re(x_{3})  - \frac{1}{2} \im( x_{2})i\end{equation}

Now, we solve equations \eqref{eqn:last1}, \eqref{eqn:last2}, \eqref{eqn:last3}, and \eqref{eqn:last4} to express $\lambda_1 \alpha, \beta, x_2,$ and $ x_3$ in terms of $y_3$.  We find:

$$\begin{bmatrix}
    \lambda_1 \alpha\\ \beta\\ \re (x_2) \\  \im (x_2) \\ \re (x_3)  \\ \im (x_3)\end{bmatrix} = \begin{bmatrix} 1 - \sqrt{2}\im (y_3) \\ -1 + 2\sqrt{2}\im (y_3) \\ \sqrt{2}\re (y_3) \\ \frac{1}{2}-\frac{3\sqrt{2}}{2}\im (y_3) \\ -\re (y_3) \\ -\frac{\sqrt{2}}{2}+2\im (y_3)\end{bmatrix}.$$

    Substituting these into equation \eqref{eqn:4.9} and taking the imaginary part, we find $$-\frac{\sqrt{2}}{2} + 2\im(y_3) = -\sqrt{2}\re(y_3)^2 + \frac{\im(y_3)}{2} - \frac{3\sqrt{2}\im(y_3)^2}{2}.$$  Using the fact that $\beta = 1-|y_3|^2$ to express $\re(y_3)^2$ solely in terms of $\im(y_3)$, this equation leads to $\im(y_3)\in \{\sqrt{2}, \frac{3\sqrt{2}}{2}\}$.

    Then equation \eqref{eqn:4.8}, when solved for $\re(y_3)^2$, gives $$\re(y_3)^2 = 1-\left(-1+2\sqrt{2} \im(y_3)\right) - \im(y_3)^2 \in \{-17/2,-4 \}.$$  This contradicts the fact that $\re(y_3)$ is a real number, complteing the proof that there are no type (iv) zero-curvature planes. 
    
    \end{proof}

%Substituting for $x_3$ and $y_3$ in $x_j = -ix_2y_j$ for $j = 3$, we can solve for $\re (y_3) $ and $\im (y_3)$, eventually showing that in all cases $y_4 = y_5 = 0$ by equations 3 and 4 together with $x_j = -ix_2y_j$ for $j = 4, j = 5$. Finally, we use $\beta = 1 - |\vec{y}|^2$ and our solution for $\beta$ above to find a contradiction for all valid choices of $y_3$. 

% Our goal is to show that the manifolds created by the action (insert name of thing here) are positively curved on a full-measure set; that is, they are almost positively curved. %I'm sure we will define almost positive curvature and not need this sentence to read like this but for now

\section{Strong inhomongeneity}\label{sec:top}

In this section, we prove that in each dimension of the form $4n-1\geq 23$, there are infinitely many generalized Eschenburg spaces that simultaneously are strongly inhomogeneous and admit metrics of almost positive sectional curvature.

\begin{lemma}\label{lem:s1homotopy} Suppose $M$ and $N$ are $2$-connected CW complexes  each admitting a free circle action.  Assume that the orbit spaces $M/S^1$ and $N/S^1$ are homotopy equivalent.  Then $M$ and $N$ are homotopy equivalent.
\end{lemma}

\begin{proof}  Let $\pi_M,\pi_N$ denote the two projection maps and let $f:M/S^1\rightarrow N/S^1$ be a homotopy equivalence.  Let $(\xi,\pi)$ denote the pullback of the $S^1$-bundle $N\xrightarrow{\pi_N} N/S^1$ under the map $(f\circ \pi_M):M\rightarrow N/S^1$.  Being a pullback, there is a continuous map $\tilde{f}:\xi\rightarrow N$ for which $\pi_N\circ \tilde{f} = f\circ \pi_M\circ \pi$.

The Euler class of $\xi$ is an element of $H^2(M)$ which vanishes by assumption, so $\xi$ is trivial.  Thus, there is a section $s:M\rightarrow \xi$.  We therefore obtain a map $\tilde{f}\circ s:M\rightarrow N$.  We claim that this map is a homotopy equivalence.

As $M$ and $N$ are $2$-connected, by Whitehead's Theorem it is sufficient to verify that $\tilde{f}\circ s$ induces isomorphisms on the higher homotopy groups $\pi_k$ for all $k\geq 3$.

To that end, we observe that from the long exact sequence in homotopy groups associated to the bundles $M\rightarrow M/S^1$ and $N\rightarrow N/S^1$, it follows that both $\pi_M$ and $\pi_N$ induce isomorphisms on $\pi_k$ for all $k\geq 3$.  Moreover, $f$, being a homotopy equivalence, also induces an isomorphism for all $k\geq 3$.  But then for any $k\geq 3$, $(\tilde{f}\circ s)_\ast = (\pi_N)_\ast^{-1}\circ f_\ast \circ (\pi_M)_\ast:\pi_k(M)\rightarrow \pi_k(N)$, so $(\tilde{f}\circ s)_\ast$ is a composition of isomorphisms and hence is an isomorphism.

\end{proof}

\begin{proposition}\label{prop:inequivalent} Suppose that $\mathcal{E}_{p,q}$ is homotopy equivalent to a homogeneous space $G/H$ where no proper normal subgroup of $G$ acts transitively on $G/H$.  If $\dim \mathcal{E}_{p,q}\geq 23$, then, up to finite cover, $(G,H) = (SU(n+1), SU(n-1)\times S^1)$.

\end{proposition}

\begin{proof}We first note that $\pi_1(\mathcal{E}) = 0$ and $\pi_2(\mathcal{E})\cong \mathbb{Z}$, which follows easily from the long exact sequence in homotopy groups associated to the bundle $S^1\times U(n-1)\rightarrow U(n+1)\rightarrow \mathcal{E}$.

Suppose that $\mathcal{E}=\mathcal{E}_{\vec{p},\vec{q}}$ is homotopy equivalent to such a $G/H$.  Then $H^2(G/H)\cong \mathbb{Z}$.  Let $N$ denote the total space of a principal $S^1$-bundle over $G/H$ whose Euler class generates $H^2(G/H)$.  From \cite{PS}, the $G$-action on $G/H$ admits a commuting lift to $N$.  Then the $G\times S^1-$action on $N$ is clearly transitive.  But, in fact, the $G$-action on $N$ is transitive.  To see this, let $H'$ denote the isotropy group of the $G\times S^1$-action on $N$ at some point.   Then $H'$ acts on $G\times S^1$ by right multiplication by inverses and the quotient is $(G\times S^1)/H'\cong N$.  Since $N$ is simply connected, the projection of the $H'$ action to the $S^1$ factor of $G\times S^1$ must be transitive.  Let $H''$ denote the isotropy group of this action at the identity $1\in S^1$.  Then, from \cite[Lemma 1.3]{KZ}, we conclude that $N\cong (G\times S^1)/H' \cong G/H''$, so the $G$-action on $N$ is transitive.

Viewing $\mathcal{E}$ as an $S^1$-quotient of $U(n+1)/U(n-1)$ and applying Lemma \ref{lem:s1homotopy} to $\mathcal{E}$ and $G/H$, we deduce that $U(n+1)/U(n-1)$ is homotopy equivalent to $N= G/H''.$  In particular, $G/H''$ is at least $10$-connected.

From \cite[Theorem 4 pg. 268 and Theorem 5 pg. 260]{On}, we know that $G$ is either simple or $G = G_1\times G_2$ is a product of two simple groups.

Assume first that $G$ is simple.  We observe that the rational homotopy groups of $U(n+1)/U(n-1)$ are non-trivial except in degrees $2n-1$ and $2n+1$, where they are isomorphic to $\mathbb{Q}$.  Further, since the even rational homotopy groups of $U(n+1)/U(n-1)$ vanish, the same is true of $N$, so $H''$ has corank $2$ in $G$.  Since $G$ must have non-trivial rational homotopy groups in consecutive odd degrees and $\dim G/H''\geq 24$, \cite[Table 11, pg. 270]{On} implies that $(G,H'')$ is one of $(SU(k),SU(k-2))$, $(Spin(4k+2), Spin(4k-1))$, $(\mathbf{E}_6,\mathbf{F}_4)$.  In the latter case, we have $\pi_9(\mathbf{E}_6/\mathbf{F}_4)\otimes\mathbb{Q}\neq 0$ owing to the fact that $\pi_9(\mathbf{E}_6)$ is rationally non-trivial while $\pi_9(\mathbf{F}_4)$ is finite.  The space $Spin(4k+2)/Spin(4k-3)$ is a bundle over $Spin(4k+2)/Spin(4k+1)\cong S^{4k+1}$ with fiber $Spin(4k+1)/Spin(4k-1)$, which is diffeomorphic to the unit tangent bundle of $S^{4k}$.  Since the unit tangent bundle is rationally $S^{8k-1}$, we find that $Spin(4k+2)/Spin(4k-1)$ has the rational homotopy groups of $S^{8k-1}\times S^{4k+1}$.  These degrees are consecutive odd numbers if and only if $k=1$, when the dimension is $12 < 23$.  Thus, the assumption that $G$ is simple leads to the conclusion that $(G,H'') = (SU(k), SU(k-2)$ for some $k$.  Dimension counting implies $k = n+1$ and then that $(G,H) = (SU(n+1), SU(n-1)\times S^1)$ as claimed.

It remains to show that the case where $G = G_1\times G_2$ cannot arise.  So, assume it does.  Since $\dim \mathcal{E}\geq 23> 15$ , it follows from \cite[Theorem 3.6]{HS} that $U(n+1)/U(n-1)$ is not homotopy equivalent to a product, which implies that there must be a factor $H'''$ of (a cover of) $H''$ which projects non-trivially to both factors of $G$.  From inspection of all six cases of \cite[Theorem 5, pg. 269]{On}, $H'''$ is, up to cover, isomorphic to $S^1$, to $SU(2)$, or to $SU(2)\times SU(2)$.  From \cite[Lemma]{KZ}, we may assume the projection of this factor to both factors of $G$ is not surjective.

If $H'''$ is $S^1$, then the long exact sequence in homotopy groups implies $\pi_2(N)\neq 0$, which is a contradiction, so case 1 is ruled out.  The assumption that $\dim \mathcal{E}\geq 23$ rules out cases 5 and 6, where $H''' = SU(2)^2$, because in these cases, $G$ has dimension at most $28$ so $G/H$ has dimension at most $22$.

Cases 3 and 4 are ruled out because in these cases, $\pi_7(G/ H''')$ is infinite, coming from the second factor of $G$.

We finally handle case 2.  If both $G_1/H_1$ and $G_2/H_2$ have trivial $\pi_3$, then the long exact sequence in homotopy groups implies $G/H$ has non-trivial $\pi_4$, which is a contradiction.  But if $G_1/H_1$ is not $3$-connected, then $H_1$ is a torus or trivial.  If it is a torus, then $\pi_2(N)\neq 0$, which is a contradiction.  If it is trivial, the projection of $H'''$ to $G_1$ is surjective, which again is a contradiction.
\end{proof}

It follows from Proposition \ref{prop:inequivalent} that in order to show $\mathcal{E}_{p,q}$ is strongly inhomogeneous, it is sufficient to show $\mathcal{E}_{p,q}$ is not homotopy equivalent to any space of the form $\mathcal{E}_{0,q_1,q_2}$.

\begin{proposition}\label{prop:topology}  Assume $\ell:= p \frac{q_1^n - q_2^n}{q_1-q_2} - \frac{q_1^{n+1}-q_2^{n+1}}{q_1-q_2}$ is non-zero. Then the $ 4n-1$-dimensional space $\mathcal{E}_{p,q_1,q_2}$ has $H^{2n}(\mathcal{E}_{p,q_1,q_2})\cong \mathbb{Z}/\ell\mathbb{Z}$.

\end{proposition}

\begin{note}  If $q_1 = q_2=1$, the notation $\frac{q_1^n-q_2^n}{q_1-q_2}$ is properly interpreted as $\sum_{i=0}^{n-1} q_1^i q_2^{n-1-i} = n$.
\end{note}

\begin{proof} 

We will follow the standard procedure for computing the cohomology groups of homogeneous spaces and biquotients that was developed by Eschenburg \cite{Esc} and is described, e.g., in \cite[Section 2.3]{Dev3}.  To set up notation, for any Lie group $K$, we let $BK$ denote the classifying space of $K$.  Let $T_G$ and $T_H$ denote the standard maximal tori in $G = U(n+1)$ and $H = S^1\times U(n-1)$ respectively, so $T_G$ and $T_H$ consist of diagonal matrices.   We let $h=(h_1,h_2):H\rightarrow G\times G$ define the action yielding the generalized Eschenburg space, so $h_1(z,C) = \diag(z^p,1,...,1)$ and $h_2(z,C) = \diag(z^{q_1}, z^{q_2}, C)$.

We write $H^\ast(BT_G)\cong \mathbb{Z}[x_1,...,x_{n+1}]$ and $H^\ast(BT_H)\cong \mathbb{Z}[z, y_3,...,y_{n+1}]$, where $z$ corresponds to the circle factor of $H$ while the $y_i$ correspond to the diagonal matrices in $U(n-1)$.

Clearly, then, $h_1^\ast(x_1) = pz$ while $h_1^\ast(x_i) = 0$ for $i > 1$.  Similarly, $h_2^\ast(x_i) = q_i z$ for $i=1,2$ while $h_2^\ast(x_i) = y_i$ for $i\geq 3$.

We identify $H^\ast(BG\times BG)$ with the subalgebra $\mathbb{Z}[\sigma_j(x_i)\otimes 1, 1\otimes \sigma_j(x_i)]$ of $H^\ast(BT_G\times BT_G)\cong \mathbb{Z}[x_i\otimes 1, 1\otimes x_i]$, where $\sigma_i$ denotes the $i$th elementary symmetric polynomial.  We can similarly identify $H^\ast(BH)$ with $\mathbb{Z}[z,\sigma_j(y_i)]\subseteq H^\ast(BT_H)$.

Then we find that $Bh_1^\ast(\sigma_1(x_i)) = pz$ while $Bh_1^\ast(\sigma_j(x_i)) = 0$ for $j>1$.  We also compute that $Bh_2^\ast(\sigma_j(x_i)) = q_1q_2 z^2 \sigma_{j-2}(y_i) + (q_1+q_2)z \sigma_{j-1}(y_i) + \sigma_j(y_i)$, where we use the convention that $\sigma_j$ is identically zero for $j <0$.

Having computed $Bh^\ast$, we can compute the differentials in the Serre spectral sequence for the bundle $G\rightarrow G\bq H\rightarrow BH$ as follows.  One chooses elements $w_{2j-1}$ in the cohomology groups of $U(n+1)$ in such a way that $dw_{2j-1} = \sigma_{j}(x_i)$ in the spectral sequence for the universal principal $U(n+1)$-bundle $U(n+1)\rightarrow EU(n+1)\rightarrow BU(n+1)$.  Then, for the bundle $G\rightarrow G\bq H\rightarrow BH$, the $w_{2j-1}$ are  transgressive and one has $dw_{2j-1}= Bh_1^\ast(\sigma_j(x_i)) - Bh_2^\ast(\sigma_j(x_i))$ on the $2j$-th page.
%\footnote{page?? {\color{blue} Spectral sequences consists of a triply countably indexed list of groups, denote $E_r^{p,q}$, where $p,q\in \mathbb{Z}$ and $r\in \mathbb{N}$.  If you fix $r$, the infinite collection of groups can be plotted on a grid, and is called the $r$-th page of the spectral sequence.  There is an inductive procedure for generating the $(r+1)$th page from the $r$th page, and the $\infty$-th page contains lots of useful information}}

For $2j-1< 2n-1$, the differential involves the non-zero term $\sigma_j(y_i)$, while for $k < j$, $dw_{2k-1}$ does not involve this term.  It follows that on the $E_{2n-1}$-th page (which is equal to the $E_{2n}$-th page), all entries between rows $1$  and $2n-2$ (inclusive) vanish.  In addition, the $2n$-th row is also trivial.  Moreover, $E_{2n-2}^{2j,0}\cong \mathbb{Z}$ and is generated by $z^j$.

The differential $dw_{2n-1}$ has the form $dw_{2n-1} = \ell z^j \in E_{2n}^{2n,0}$ for some $a\in \mathbb{Z}$.  We will shortly prove that $\ell = p \frac{q_1^n - q_2^n}{q_1-q_2} - \frac{q_1^{n+1}-q_2^{n+1}}{q_1-q_2}$.  Believing this momentarily, observe that if $\ell\neq 0$, it follows that $H^2(\mathcal{E}_{p,q})$ is finite cyclic of order $|\ell|$, establishing the proposition.

To determine $\ell$, we use the differentials $dw_1 = (p-(q_1+q_2)) z$ and $dw_{2j-1} = - ( q_1 q_2 \sigma_{j-2}(y_i) z^2 + (q_1+q_2)\sigma_{j-1}(y_i)z + \sigma_j(y_i)$ for $j>1$ and a straightforward induction to establishes the formula $\sigma_j(z_i) = (-1)^{j+1} \left( p \frac{q_1^j-q_2^j}{q_1- q_2} - \frac{q_1^{j+1} - q_2^{j+1}}{q_1-q_2}\right)\in H^{2j}(\mathcal{E}_{p,q})$.  Thus, $$\ell =  p \frac{q_1^n-q_2^n}{q_1- q_2} - \frac{q_1^{n+1} - q_2^{n+1}}{q_1-q_2},$$ as claimed.

\end{proof}

\begin{theorem}  Suppose $n\geq 6$.  Then there are infinitely many almost positively curved strongly inhomogeneous spaces of the form $\mathcal{E}^{4n-1}_{p,q_1,q_2}$.
\end{theorem}

\begin{proof}  We will show that for each $n\geq 6$, there are infinitely many choices of $p$ for which the generalized Eschenburg space $\mathcal{E}_{p,0,1}$ is strongly inhomogeneous.  From Proposition \ref{prop:inequivalent}, in order to prove that $\mathcal{E}_{p,q_1,q_2}$ is strongly inhomogeneous, it is sufficient to show it is not homotopy equivalent to any generalized Eschenburg space $\mathcal{E}_{0,r_1,r_2}$.

We begin with the case where $n$ is even.  In this case, Proposition \ref{prop:topology} implies that we observe that for any homogeneous generalized Eschenburg space $\mathcal{E}_{0,r_1,r_2}$, the order of $H^{2n}$ is given by $\ell:=-\frac{r_1^{n+1}-r_2^{n+1}}{r_1-r_2}$.  We claim this is odd. To determine the parity of $\ell$, we consider the possible parities of $r_1$ and $r_2$.  If both $r_1$ and $r_2$ are odd, then $\ell$ is a sum of $n+1$ odd numbers.  As $n$ is even, we find $\ell$ is odd in this case.  Next, observe that because $r_1$ and $r_2$ are relatively prime, at most one is even.  Without loss of generality, we may assume that $r_1$ is even and $r_2$ is odd.  Then all terms in the sum defining $\ell$ are even, except for $r_2^n$, which is odd.  Thus $\ell$ is odd in this case as well.

Thus, to conclude the case where $n$ is even, we will find almost positively curved examples for which $H^{2n}$ has even order.  We consider $\mathcal{E}_{p,q_1,q_2}$ with $(p,q_1,q_2) = (2s+1,1,2s-1)$.  These are admissible for every $s$ because $2s$ and $2s-1$ are relatively prime.  From Theorem \ref{thm:main}, they admit almost positive curvature whenever $s\geq 0$.  Moreover, from Proposition \ref{prop:topology}, we see that $H^{2n}(\mathcal{E}_{p,q_1,q_2})$ has order $$(2s+1) \frac{(2s)^n - 1}{2s-1} - \frac{(2s)^{n+1}-1}{2s-1}.$$  This is even for any $s$, and thus each such $\mathcal{E}_{p,q_1,q_2}$  is strongly inhomogeneous.

\bigskip

We next turn to the case where $n$ is odd.  We first claim that $|\ell|$ cannot be equal to a prime that is congruent to $3$ mod $4$.

Because $n$ is odd, $|\ell|$ factors as $|q_1+q_2|\sum_{i=0}^{(n-1)/2} q_1^{n-1-2i}q_2^{2i}$.  Observe that $q_1^{n-1-2i}q_2^{2i} = (q_1^{\frac{n-1}{2} - i} q_2^i)^2$, so the second factor is a sum of squares.  Assuming $|\ell|$ is prime, any solution to $|\ell| = |q_1+q_2|\sum_{i=0}^{(n-1)/2} q_1^{n-1-2i}q_2^{2i}$ must have either $|q_1+q_2|=|\ell|$ and $\sum_{i=0}^{(n-1)/2} q_1^{n-1-2i}q_2^{2i}=1$, or  $|q_1+q_2|=1$ and $\sum_{i=0}^{(n-1)/2} q_1^{n-1-2i}q_2^{2i} = |\ell|$.

The first case can never occur: if $\sum_{i=0}^{(n-1)/2} q_1^{n-1-2i}q_2^{2i}=1$, then because it is a sum of non-negative terms we must have that one of $q_1$ and $q_2$ is zero while the other is $\pm 1$.  This, then, implies that $|\ell|=|q_1+q_2| = 1$, contradicting the fact that $|\ell|$ is prime.

For the second case, the equation $|q_1+q_2|=1$ implies $q_1$ and $q_2$ have opposite parities.  We assume without loss of generality that $q_2$ is even and $q_1$ is odd.  Then $q_2^i$ is divisible by $4$ for all $i\geq 2$.  Thus, computing $\sum_{i=0}^{(n-1)/2} q_1^{n-1-2i}q_2^{2i}\pmod{4}$, only the $i=0$ term remains.  That is, $\sum_{i=0}^{(n-1)/2} q_1^{n-1-2i}q_2^{2i}\equiv q_1^{n-1}\pmod{4}$.  Since $n-1$ is even and $q_1$ is odd, we find that $q_1^{n-1}\equiv 1\pmod{4}$.  Thus, this case cannot occur if $|\ell| \equiv 3\pmod{4}$.

It remains to find infinitely many almost positively curved $\mathcal{E}_{p,q}$ for which the order of $H^{2n}$ is a prime congruent to $3\pmod{4}$.  For these, we consider $\mathcal{E}_{p,q_1,q_2}$ with $(p,q_1,q_2) =(p,1,1)$ with $p\geq 2$.  These triples are admissible for any integer $p$, and they have almost positive curvature for any $p \geq 2$ by Theorem \ref{thm:main}.  By Proposition  \ref{prop:topology}, the order of $H^{2n}$ is $|pn - (n+1)| = (p-1)n - 1$.  By Dirichlet's theorem, since $\gcd(-1,4n) = 1$, there are infinitely many primes of the form $4kn-1$ with $k\in \mathbb{Z}$.  For each such $k$, letting $p = 4k+1$, it follows that $\mathcal{E}_{p,1,1}$ is strongly inhomogeneous.
\end{proof}

\bibliography{sources}
\bibliographystyle{plain}

\end{document}